\documentclass[11pt]{amsart} 

\usepackage{amsthm,amsbsy,amsfonts,amssymb,amscd}
\usepackage[mathscr]{euscript}

\usepackage{enumerate} 

\usepackage{tikz-cd} 
\usepackage[all]{xy}
\xyoption{knot}
\xyoption{poly}

\usepackage{hyperref}
\usepackage[margin=1.05in]{geometry}
\newcommand{\cF}{{\mathcal{F}}}

\newcommand{\cG}{\mathcal{G}}

\newcommand{\RF}{\mathbf{R}}
\newcommand{\EE}{\mathbf{E}}
\newcommand{\wdd}{\mathrm{wd}}
\newcommand{\htt}{\mathrm{ht}}

\newcommand{\Pq}{\mathbf{P}_q}
\newcommand{\oPq}{\overline{\mathbf{P}}_q}
\newcommand{\Eq}{\mathbf{E}_q}
\newcommand{\oEq}{\overline{\mathbf{E}}_q}

\newcommand{\TL}{{\mathrm{TL}}}

\newcommand{\qss}{{\scriptstyle q}}
\newcommand{\qmineen}{{\scriptstyle q\minus1}}
\newcommand{\qmintwee}{{\scriptstyle q\minus2}}
\newcommand{\rmineen}{{\scriptstyle r\minus1}}
\newcommand{\rmintwee}{{\scriptstyle r\minus2}}
\newcommand{\qmindrie}{{\scriptstyle q\minus3}}
\newcommand{\qminvier}{{\scriptstyle q\minus4}}

\newcommand{\qpluseen}{{\scriptstyle q\plu1}}
\newcommand{\rpluseen}{{\scriptstyle r\plu1}}
\newcommand{\rss}{{\scriptstyle r}}

\newcommand{\RR}{\mathscr{R}}
\newcommand{\AAAA}{\mathscr{A}}

\newcommand{\Hom}{\mathrm{Hom}}

\newcommand{\minus}{\scalebox{0.9}{{\rm -}}}
\newcommand{\plu}{\scalebox{0.6}{{\rm +}}}

\newcommand{\End}{\mathrm{End}}

\newcommand{\dfs}{{/\kern-2pt/}}

\newcommand{\mk}{\Bbbk}
\newcommand{\cC}{\mathcal{C}}

\newcommand{\fh}{\mathfrak{h}}

\newcommand{\mN}{\mathbb{N}}

\newcommand{\mZ}{\mathbb{Z}}

\numberwithin{equation}{section}

\newcommand{\con}{{\rm con}}
\newcommand{\LL}{{\mathbf{L}}}

\newtheoremstyle{notes} {} {} {} {} {\bfseries} {.} {.5em} {}

\theoremstyle{plain}

\newtheorem{prop}[subsubsection]{Proposition}

\newtheorem{lemma}[subsubsection]{Lemma}

\newtheorem{cor}[subsubsection]{Corollary}

\newtheorem{thm}[subsubsection]{Theorem}

\newtheorem{thmA}{Theorem}

\theoremstyle{remark}

\newtheorem{rem}[subsubsection]{Remark} 

\newtheorem{ddef}[subsubsection]{Definition} 

\swapnumbers

\usepackage{etoolbox}

\usepackage[enableskew]{youngtab}

\makeatletter
\pretocmd{\appendix}{\addtocontents{toc}{\protect\addvspace{10\p@}}}{}{}
\makeatother

\theoremstyle{remark}

\newtheorem{ex}[subsubsection]{Example}

\newtheoremstyle{construction} {} {} {} {} {\bfseries} { } {0pt} {}

\theoremstyle{construction}

\title[The periplectic Brauer algebra II]{The periplectic Brauer algebra II: decomposition multiplicities}

\newcommand{\res}{{\rm Res}}

\newcommand{\mineen}{\mbox{-1}}
\newcommand{\mintwee}{\mbox{-2}}
\newcommand{\leeg}{\mbox{{\color{white} Test}}}


\begin{document} 
\date{} 
\begin{abstract}
We determine the Jordan-H\"older decomposition multiplicities of projective and cell modules over periplectic Brauer algebras in characteristic zero. These are obtained by developing the combinatorics of certain skew Young diagrams. We also establish a useful relationship with the Kazhdan-Lusztig multiplicities of the periplectic Lie supergroup.
\end{abstract}

\author{Kevin Coulembier}
\address{K.C.: School of Mathematics \& Statistics, University of Sydney, NSW 2006, Australia}
\email{kevin.coulembier@sydney.edu.au}

\author{Michael Ehrig}
\address{M.E.: School of Mathematics \& Statistics, University of Sydney, NSW 2006, Australia}
\email{michael.ehrig@sydney.edu.au}

\maketitle 


\section*{Introduction} 

The periplectic Brauer algebra~$A_r$ was introduced by Moon in \cite{Moon} in the study of invariant theory of the periplectic Lie superalgebra. More recently, Kujawa and Tharp developed a diagram calculus for this algebra, which is a non-trivial adaptation of the diagrammatic approach to the Brauer algebra. This was exploited by the first author in \cite{PB1} in order to  determine the blocks of~$A_r$ over fields of characteristic zero. An important tool was the study of the periplectic Brauer algebra in the framework of standardly based algebras of \cite{JieDu}. This study also showed that, excluding some exceptional cases of low dimension or low characteristic of the base field, $A_r$ is either quasi-hereditary or admits a quasi-hereditary 1-cover.

As a standardly based algebra, $A_r$ admits cell modules, which are precisely the standard modules when $A_r$ is quasi-hereditary. In this paper, we calculate the Jordan-H\"older decomposition multiplicities of these cell modules in characteristic zero. The cell modules are labelled by partitions and the simple modules by non-empty partitions. The multiplicities for $A_r$ do not depend explicitly on~$r$ and our main result is stated in the following theorem.
\begin{thmA}
Let $\Gamma$ be the set of skew Young diagrams which are either zero or such that their maximal connected outer rim hooks satisfy the property that
\begin{itemize}
\item the width of the hook is one bigger than the height;
\item no box in the hook is strictly above the line going through the diagonal of the left most box (for the diagonal for which such a condition makes sense);
\end{itemize}
and such that the skew Young diagram obtained after removing the outer rim hooks is again in~$\Gamma$. We have
$$[W(\lambda):L(\mu)]=\begin{cases}1 &\mbox{if }\; \lambda\subseteq\mu\;\mbox{ and }\; \mu/\lambda\in\Gamma\\
0 &\mbox{otherwise}.
\end{cases}$$
\end{thmA}
We also obtain a description of the Cartan decomposition matrix of~$A_r$.
\begin{thmA}
With $\Gamma$ the set of skew Young diagrams of Theorem~1 and $\Gamma'$ the set consisting of the conjugates of the diagrams in $\Gamma$, we have
$$[P(\nu):L(\mu)]=\begin{cases}1 &\mbox{if there exists~$\lambda$ with}\; \nu\supseteq \lambda\subseteq\mu, \;\mu/\lambda\in\Gamma \;\mbox{ and }\; \nu/\lambda\in\Gamma'\\
0 &\mbox{otherwise}.
\end{cases}$$
\end{thmA}

Our description of the decomposition multiplicities is very different from the corresponding result for the ordinary Brauer algebra in~\cite{BrMult, Martin}. In the latter case, the decomposition multiplicities are given in terms of parabolic Kazhdan-Lusztig polynomials of type $D$ with respect to a maximal parabolic subalgebra of type $A$. Also the proof is very different. The most intricate part of our proof is actually proving equivalence between several descriptions of the set of skew Young diagrams~$\Gamma$.
One essential tool remains, as for Brauer algebras, the restriction of modules from $A_r$ to $A_{r\minus 1}$ in connection with the action of a Jucys-Murphy type element of~$A_r$. However, the resulting information is far less conclusive and elegant than the `translation principle' used in~\cite{BrMult, Martin}. Nevertheless, we demonstrate how the decategorifications of the restriction functors relate to an infinite Temperley-Lieb algebra.

Recently, the decomposition numbers for the periplectic Lie superalgebra were determined in \cite{gang}. These are described in terms of an arrow diagram calculus. In an appendix we rewrite the result in Theorem 1 in terms of a very similar arrow diagram calculus. As a consequence, we find an intimate relation between the two types of decomposition multiplicities. We rely on this to prove that the non-zero entries in the Cartan decomposition matrix of~$A_r$ in Theorem~2 must be $1$, by using the corresponding result in~\cite{gang}.

The paper is organised as follows. In Section~\ref{Prel} we recall some terminology of partitions and properties of periplectic Brauer algebras. In Section~\ref{SecRes} we derive some properties of the interplay of the restriction functor and a Jucys-Murphy element, as introduced in~\cite{PB1}. Section~\ref{Skew} is a purely combinatorial study of the set $\Gamma$ of skew Young diagrams. In~Section~\ref{SecCell} we combine the representation theoretic results of Section~\ref{SecRes} with the combinatorial ones in Section~\ref{Skew} in order to determine the decomposition multiplicities for periplectic Brauer algebras. Finally, in Appendix A, we demonstrate that the decomposition multiplicities can also be described using the arrow diagram calculus of \cite{gang}.

Jonathan Kujawa has informed us that he and Ben Tharp independently obtained a description of the cell multiplicities of the periplectic Brauer algebra, which will appear soon.

\section{Preliminaries}\label{Prel}
For the entire paper we fix an algebraically closed field~$\mk$ of characteristic zero. We use the canonical inclusion~$\mZ\subset\mk$ of unital rings.

\subsection{Partitions}
\subsubsection{}
We will identify a partition with its Young diagram, using English notation. For instance, the partition~$(3,1)$ is represented by the diagram {$\begin{minipage}{.85cm} \scalebox{.6}{\yng(3,1)} \end{minipage}$}. Each box or node in the diagram has coordinates $(i,j)$, meaning that the box is in row $i$ and column~$j$. The above diagram has boxes with coordinates $(1,1)$, $(1,2)$, $(1,3)$ and~$(2,1)$.
The {\em content} of a box~$b$ in position~$(i,j)$ in a Young diagram is $\con(b):=j-i\in\mZ$. The content of each box in the Young diagram of~$(3,1)$ is displayed as {$\begin{minipage}{.9cm} \scalebox{.6}{\young(012,\mineen)} \end{minipage}$}. Any box with content $q$ will be referred to as a~$q$-box. We will occasionally also need the value $i+j$ for a box in position~$(i,j)$. We refer to that value as the {\em anticontent} of the box.

\subsubsection{}

For any partition~$\lambda$ we define the set of partitions~$\AAAA(\lambda)$, resp. $\RR(\lambda)$, containing all partitions which can be obtained from~$\lambda$ by adding an addable box, resp. removing a removable box. For any $q\in\mZ$, we consider the subset $\AAAA(\lambda)_q$ of~$\AAAA(\lambda)$, consisting of the partitions obtained by adding a~$q\minus 1$-box. Similarly, the partitions in~$\RR(\lambda)_q$ are obtained by removing a~$q$-box. With this convention we have
$$\mu\in\AAAA(\lambda)_q\quad\Leftrightarrow \quad \lambda\in \RR(\mu)_{q\minus 1}.$$
Obviously the sets $\AAAA(\lambda)_q$ and~$\RR(\lambda)_q$ are either empty or contain precisely one element.

\subsubsection{}We denote the empty partition by~$\varnothing$. On the other hand, the empty set will be denoted by~$\emptyset$. This means that we have 

$$\RR({\Box})_q=\begin{cases}\{\varnothing\}&\mbox{if } \;q=0,\\
\;\,\emptyset&\mbox{if }\;q\not=0.\end{cases}
$$ 

\subsection{The periplectic Brauer algebra}\label{IntroAr} The periplectic Brauer algebra~$A_r$ was introduced in~\cite{Moon}. In \cite{Kujawa}, a diagrammatic description of the algebra was developed. In particular, $A_r$ has a~$\mk$-basis of ordinary Brauer diagrams, but multiplication is complicated by appearance of minus signs. In the current paper we do not need the diagrammatic description directly, although we rely on results of \cite{PB1} which heavily exploited the diagrammatic rules of~\cite{Kujawa}. Hence we do not repeat the diagrammatic description here. It was proved in \cite[Theorem~4.3.1]{Kujawa} that the simple $A_r$-modules $L_r(\mu)$ are labelled by the set of partitions
$$\Lambda_r=\{\mu\vdash r-2i\,|\, 0\le i < r/2\}.$$

In \cite[Theorem~3]{PB1}, it was proved that~$A_r$ admits an interesting standardly based structure, where the latter is a generalisation of cellular algebras introduced in~\cite{JieDu}. In particular, $A_r$ has {\em cell modules} $W_r(\lambda)$, labelled by partitions in the set
$$\LL_r=\{\lambda\vdash r-2i\,|\, 0\le i \le r/2\}.$$
When $\lambda\not=\varnothing$, the module $W_r(\lambda)$ has simple top $L_r(\lambda)$ and the radical of~$W_r(\lambda)$ only has simple constituents $L_r(\nu)$ with $|\nu|>|\lambda|$. 
The projective $A_r$-modules admit a filtration with sections given by cell modules. Moreover, the multiplicities in this filtration satisfy the following twisted Humphreys-BGG reciprocity relation
\begin{equation}\label{eqBGG}(P_r(\mu):W_r(\lambda))\;=\;[W_r(\lambda'):L_r(\mu')],\quad\mbox{for all $\lambda\in\Lambda_r$ and~$\mu\in\LL_r$,}\end{equation}
where $\lambda'$ denotes the conjugate of a partition~$\lambda$. In particular, determining the Jordan-H\"older multiplicities of the cell modules, henceforth referred to as {\em cell multiplicities}, also determines the Jordan-H\"older multiplicities of the indecomposable projective modules. Concretely, the Cartan decomposition matrix can be expressed as \begin{equation}\label{eqPL}[P_r(\nu):L_r(\mu)]\;\,=\; \,\sum_{\lambda\in\LL_r}[W_r(\lambda):L_r(\mu)]\,[W_r(\lambda'):L_r(\nu')],\end{equation}
for $\nu,\mu\in\Lambda_r$.

Note that the above discussion implies in particular that, when $r$ is odd (and hence $\LL_r=\Lambda_r$), the algebra~$A_r$ is quasi-hereditary, with standard modules given by cell modules.

\subsection{Some preliminary results on cell multiplicities}
We will always assume $r\in\mZ_{\ge 2}$. By \cite[equation~(4.11)]{PB1}, we have the following 
reinterpretation of a lemma in \cite{PB1}:
\begin{lemma}[Lemma~7.2.2 in~\cite{PB1}]\label{2Lem}
For $\lambda\in \LL_r$ and~$\mu\in\Lambda_r$, we have
\begin{enumerate}[(i)]
\item $[W_r(\lambda):L_r(\mu)]=0 \text{ unless } \lambda\subseteq\mu$, and
\item $[W_r(\lambda):L_r(\mu)]\,=\, [W_i(\lambda):L_i(\mu)],\text{ for }  i=|\mu|.$
\end{enumerate}
\end{lemma}

Combining some results in \cite{PB1} also yields the following statement.
\begin{prop}\label{PropDiff2}
For $\lambda\vdash r\minus 2$ and~$\mu\in\Lambda_r$, we have
$$[W_r(\lambda):L_r(\mu)]=\begin{cases}1&\mbox{if $\mu$ is obtained from~$\lambda$ by adding a rim 2-hook {$\begin{minipage}{.65cm} \scalebox{.7}{\yng(2)} \end{minipage}$}},\\ 1&\mbox{if}\;\,\mu=\lambda,\\ 0&\mbox{otherwise.}\end{cases}$$
\end{prop}
\begin{proof}
The case where $\mu$ is obtained from~$\lambda$ by adding a rim 2-hook {$\begin{minipage}{.6cm} \scalebox{.7}{\yng(2)} \end{minipage}$} is \cite[Proposition~7.2.6]{PB1}. The case $\lambda=\mu$ is discussed in Section~\ref{IntroAr}. It thus only remains to prove the vanishing of other multiplicities for $\mu\not=\lambda$. By Lemma~\ref{2Lem}(i), it suffices to consider the case $\lambda\subset \mu$, so $\mu\vdash r$. 

Assume first that~$\mu$ is obtained from~$\lambda$ by adding two boxes which are neither in the same row nor same column. Their contents must thus differ by at least $2$, which yields a contradiction by \cite[Corollary~6.2.7]{PB1}.

Finally assume that~$\mu$ is obtained from~$\lambda$ by adding a rim 2-hook {$\begin{minipage}{.35cm} \scalebox{.6}{\yng(1,1)} \end{minipage}$}. In this case, the combination of \cite[Lemma~7.2.3]{PB1} and \cite[Corollary~4.3.3]{PB1} proves the vanishing, concluding the proof.
\end{proof}

\section{The restriction functor}\label{SecRes}
The algebra~$A_{r\minus 1}$ is a subalgebra of~$A_r$, see~\cite[2.1.7]{PB1}. We denote by 
$$\res_{r}\,:\, A_r\mbox{-mod}\;\to\;A_{r\minus 1}\mbox{-mod},$$ the restriction functor corresponding to this embedding $A_{r\minus 1}\subset A_r$.

In \cite[Section~6.1]{PB1} a Jucys-Murphy (JM) element $x_r\in A_r$ was introduced. By \cite[Lemma~6.1.2]{PB1}, all elements of the subalgebra~$A_{r\minus 1}$ of $A_r$ commute with $x_r$. For any $A_r$-module $M$, the $A_{r\minus 1}$-module $\res_{r}M$ thus naturally decomposes into generalised eigenspaces for $x_r$. We write
$$\res_{r}M\;=\;\bigoplus_{\alpha} M_\alpha,$$
where $M_\alpha$ is the $A_{r\minus 1}$-submodule of~$M$ on which $(x_r\,\minus\,\alpha)$ acts nilpotent.

\subsection{Restriction of cell modules}

\begin{lemma}[Corollary~5.24 in \cite{PB1}]\label{CorPB1}
For any $\lambda\in\LL_r$, we have a short exact sequence
$$0\;\to\;\bigoplus_{\mu\in \RR(\lambda)}W_{r\minus 1}(\mu)\;\to\;\res_{r}W_r(\lambda)\;\to\; \bigoplus_{\nu\in\AAAA(\lambda)}W_{r\minus 1}(\nu)\;\to\;0.$$
The left term vanishes if $\lambda=\varnothing$, the right term vanishes if $\lambda\vdash r$.
\end{lemma}

By \cite[Lemma~6.2.5]{PB1}, or \cite[5.2.5 and Theorem~6.2.2]{PB1}, we can strengthen this result to include the action of~$x_r$ as follows.
\begin{prop}\label{Propalpha}
For any $\lambda\in\LL_r$, we have $\res_{r}W_r(\lambda)\;=\;\bigoplus_{q\in\mZ} W_r(\lambda)_q$.
For any $q\in\mZ$, we have a short exact sequence
$$0\;\to\;\bigoplus_{\mu\in \RR(\lambda)_q}W_{r\minus 1}(\mu)\;\to\;W_r(\lambda)_q \;\to\; \bigoplus_{\nu\in \AAAA(\lambda)_q}W_{r\minus 1}(\nu)\;\to\;0.$$
\end{prop}
In particular, the special case $\lambda\vdash r$ yields
\begin{equation}\label{eqSC}
L_r(\lambda)_q\;=\;\begin{cases} L_{r\minus 1}(\mu)&\mbox{if}\quad \RR(\lambda)_q=\{\mu\},\\
0&\mbox{if}\quad \RR(\lambda)_q=\emptyset.
\end{cases}
\end{equation}

\begin{cor}\label{Coreta}Fix arbitrary $\lambda\in\LL_r$ and~$q\in\mZ$.
\begin{enumerate}[(i)]
\item For $\eta\in \AAAA(\lambda)_q$, we have $[W_r(\lambda)_q:L_{r\minus1}(\eta)]=1$.
\item For $\eta\in \AAAA(\lambda)_{q+ 1}$, we have $[W_r(\lambda)_{q\minus 1}:L_{r\minus1}(\eta)]=\begin{cases}1&\mbox{if }\RR(\lambda)_{q\minus 1}\not=\emptyset,\\
0&\mbox{otherwise.} \end{cases}$
\end{enumerate}
\end{cor}
\begin{proof}
Consider $\eta$ as in part~(i). Proposition~\ref{Propalpha} implies
$$[W_r(\lambda)_q:L_{r\minus1}(\eta)]\;=\;1+\sum_{\mu\in\RR(\lambda)_q}[W_{r\minus 1}(\mu):L_{r\minus 1}(\eta)].$$ If there exists~$\mu\in\RR(\lambda)_q$, then $\eta$ is obtained from~$\mu$ by adding a rim 2-hook {$\begin{minipage}{.35cm} \scalebox{.6}{\young(q,\qmineen)} \end{minipage}$}. This in turn implies $[W_{r\minus 1}(\mu):L_{r\minus 1}(\eta)]=0$ by Proposition~\ref{PropDiff2}. This proves part~(i).

Now consider $\eta$ as in part~(ii). Proposition~\ref{Propalpha} implies
\begin{equation}\label{eqqm1}[W_r(\lambda)_{q\minus 1}:L_{r\minus1}(\eta)]=\sum_{\mu\in \RR(\lambda)_{q-1}}[W_{r\minus 1}(\mu):L_{r\minus 1}(\eta)]+\sum_{\nu\in \AAAA(\lambda)_{q- 1}}[W_{r\minus 1}(\nu):L_{r\minus 1}(\eta)].\end{equation}
By Lemma~\ref{2Lem}(i), we have $[W_{r\minus 1}(\nu):L_{r\minus 1}(\eta)]=0$ unless $\nu=\eta$, for any $\nu\vdash r\minus 1$. However, as $\eta\in\AAAA(\lambda)_{q+1}$ and~$\nu\in\AAAA(\lambda)_{q\minus1}$, we find $\nu\not=\eta$. Hence, the second term on the right-hand side of~\eqref{eqqm1} vanishes. Thus if $\RR(\lambda)_{q\minus 1}=\emptyset$, the left-hand side of~\eqref{eqqm1} must indeed vanish. If $\RR(\lambda)_{q\minus 1}\not=\emptyset$, then $\RR(\lambda)_{q\minus 1}=\{\mu\}$, with $\eta$ obtained from~$\mu$ by adding the rim 2-hook {$\begin{minipage}{.65cm} \scalebox{.7}{\young(\qmineen \qss)} \end{minipage}$}. Proposition~\ref{PropDiff2} thus implies $[W_{r\minus 1}(\mu):L_{r\minus 1}(\eta)]=1$, which concludes the proof of part~(ii).
\end{proof}

\subsection{Restriction of simple modules}
The previous subsection completely determines the restriction from~$A_r$ to $A_{r\minus 1}$ of the simple cell modules $W_r(\lambda)=L_r(\lambda)$ with $\lambda\vdash r$, see equation~\eqref{eqSC}. In this section we will obtain some partial information of the restriction to $A_r$ of the simple modules $L_{r+1}(\nu)$ with $\nu\vdash r\minus 1$. These $L_{r+1}(\nu)$ are generally not cell modules.

\begin{lemma}\label{LemResS1}
Assume that~$\mu\vdash r$ has no addable $q+1$-box and~$\nu\in\RR(\mu)_q$, then
$$[L_{r+1}(\nu)_{q+1}:L_r(\mu)]=1.$$
\end{lemma}
\begin{proof}
As $\mu\in\AAAA(\nu)_{q+1}$, Corollary~\ref{Coreta}(i) implies 
$$[W_{r+1}(\nu)_{q+1}:L_r(\mu)]=1.$$
 It remains to be proved that~$L_r(\mu)$ cannot be a subquotient of~$M_{q+ 1}$, for $M$ the radical of~$W_{r+1}(\nu)$. By Proposition~\ref{PropDiff2}, the $A_{r+ 1}$-module $M$ is the direct sum of all simple modules $L_{r+1}(\lambda)=W_{r+1}(\lambda)$ with $\lambda\vdash r+1$ obtained by adding a rim 2-hook {$\begin{minipage}{.6cm} \scalebox{.7}{\yng(2)} \end{minipage}$} to $\nu$. If $L_r(\mu)$ appears in~$L_{r+1}(\lambda)_{q+ 1}$ for such a~$\lambda$, then, by equation~\eqref{eqSC}, $\mu$ can be obtained from~$\lambda$ by removing a~$q+1$-box. This is impossible as, by assumption, $\mu$ has no addable $q+1$-box.\end{proof}

\begin{rem}
Contrary to the restriction method in \cite[Section~3]{BrMult}, the above lemma does not extend to the claim $L_{r+1}(\nu)_{q+1}\cong L_r(\mu)$. It follows for instance easily that
$$[L_3(1)_0]=[L_2(2)]+[L_2(1,1)].$$
\end{rem}

\begin{lemma}\label{LemResS2}
Assume that~$\nu\in\RR(\mu)_q$ for $\mu\vdash r$ and that~$\nu$ has a removable $q\minus 1$-box, then
$$[L_{r+1}(\nu)_{q\minus 1}:L_{r}(\mu)]\,=\,1.$$
\end{lemma}
\begin{proof}
By assumption, $\mu\in\AAAA(\nu)_{q+1}$ and~$\RR(\nu)_{q\minus 1}\not=\emptyset$, so Corollary~\ref{Coreta}(ii) implies 
$$[W_{r+1}(\nu)_{q\minus 1}:L_r(\mu)]=1.$$
It remains to be proved that~$L_r(\mu)$ cannot be a subquotient of~$M_{q\minus 1}$, for $M$ the radical of~$W_{r+1}(\nu)$. By Proposition~\ref{PropDiff2}, the $A_{r+ 1}$-module $M$ is the direct sum of all simple modules $L_{r+1}(\lambda)=W_{r+1}(\lambda)$ with $\lambda\vdash r+1$ obtained by adding a rim 2-hook {$\begin{minipage}{.6cm} \scalebox{.7}{\yng(2)} \end{minipage}$} to $\nu$. If $L_r(\mu)$ appears in~$L_{r+1}(\lambda)_{q\minus 1}$ for such a~$\lambda$, then, by equation~\eqref{eqSC}, $\mu$ can be obtained from~$\lambda$ by removing a~$q\minus 1$-box. On the other hand, by working via~$\nu$, $\lambda$ is obtained from~$\mu$ by first removing a~$q$-box and then adding {$\begin{minipage}{.65cm} \scalebox{.7}{\yng(2)} \end{minipage}$}. In order for the two procedures to yield identical content, the latter rim 2-hook would have to be {$\begin{minipage}{.65cm} \scalebox{.7}{\young(\qmineen q)} \end{minipage}$}. However, for any partition, it is impossible to add {$\begin{minipage}{.6cm} \scalebox{.7}{\young(\qmineen q)} \end{minipage}$} directly after removing a~$q$-box.
\end{proof}

\subsection{Decategorification of the restriction functor}
Consider the abelian category
$$\cC_A\;=\;\bigoplus_{r\ge 2} A_r\mbox{-mod},$$
on which we have the exact functors
$$\RF:=\bigoplus_{r\ge 3}\res_r\quad\mbox{and}\quad \EE :=\bigoplus_{r\ge 4} e^{(r)}-\;=\; \bigoplus_{r\ge 4}\Hom_{A_r}(A_r e^{(r)},-).$$
Here $e^{(r)}$ is an idempotent in $A_r$, introduced as $c^\ast_{r\,\minus\,2}$ in \cite[Section~4.4]{PB1}, with the property $e^{(r)}A_re^{(r)}\cong A_{r\,\minus\,2}$. In particular, $e^{(r)}-$ is an exact functor from $A_r$-mod to $A_{r\,\minus\,2}$-mod which sends the simple module $L_r(\mu)$ to 0 if $\mu\vdash r$, or to $L_{r\,\minus\,2}(\mu)$ if $|\mu|<r$.

Decomposition of the restriction functors with respect to the eigenvalues of the JM elements~$x_r$ then yields exact functors $\RF_q$ for each $q\in \mZ$, so $\RF\cong \oplus_q \RF_q$ and $\RF_q(M)=M_q$. Let $\cG_A$ denote the Grothendieck group of $\cC_A$. The image in $\cG_A$ of a module $M$ will be denoted by~$[M]$. Similarly, $[F]\in\End_{\mk}(\cG_A)$ denotes the morphism induces by an exact functor $F$. We find the following analogue of~\cite[Corollary~4.4.6]{gang}.
\begin{prop}
For any $p,q\in\mZ$ with $|p-q|>1$, we have
$$[\RF_q]^2=0, \;\; [\RF_q][\RF_p]=[\RF_p][\RF_q]\;\;\mbox{and }\;\;[\RF_q][\RF_{q\pm 1}][\RF_q]=[\EE][\RF_q].$$
These are the defining relations of the infinite Temperley-Lieb algebra~$\TL_\infty(0)$, up to the appearance of $[\EE]$.
\end{prop}
\begin{proof}
By Section~\ref{IntroAr}, we have a basis $\{[W_r(\lambda)]\,|\, \lambda\in\Lambda_r\}$ of the Grothendieck group of $A_r$. In particular, the set  
$$\{[W_r(\lambda)]\,|\, r\in\mZ_{\ge 2}\,,\, \lambda\in\LL_r\}$$
spans $\cG_A$. Proposition~\ref{Propalpha} implies that
$$[\RF_q]([W_r(\lambda)])\;=\;\sum_{\mu\in \RR(\lambda)_q}[W_{r\minus 1}(\mu)]+\sum_{\nu\in\AAAA(\lambda)_q}[W_{r\minus 1}(\nu)],$$
where the last term is interpreted as zero when $\lambda\vdash r$. By Lemma~\ref{2Lem}(ii), we have
$$[\EE][W_r(\lambda)]=\begin{cases}[W_{r\,\minus\,2}(\lambda)]&\mbox{if }\; |\lambda|<r,\\
0 & \mbox{if }\;\lambda\vdash r.\end{cases}$$
It thus suffices to check that these equations are consistent with the proposed relations.

We cannot remove two $q$-boxes from or add two $q-1$-boxes to a partition. It is also impossible to remove a~$q-1$-box just after adding a~$q$-box, or to add a~$q$-box just after removing a~$q-1$-box. This implies that $[\RF_q]^2=0$. The other relations are similarly checked by tracking the possibilities of adding and removing boxes with the appropriate content to a partition.
\end{proof}

\section{The set of skew Young diagrams}\label{Skew}
By a skew Young diagram $\kappa$ we mean a collection of boxes which can be interpreted as the difference of a Young diagram $\mu$ and a Young diagram $\lambda\subseteq \mu$, denoted by~$\mu/\lambda$.  Any skew Young diagram has infinitely many such interpretations. Note especially that, whenever possible, we do not fix the position of the boxes in the plane in contrast to Young diagrams.

In this section, we will introduce a set of skew partitions with three different descriptions, two iterative and one in terms of decompositions into hooks. All three descriptions will be essential to prove that this set of skew partitions determines the decomposition multiplicities of periplectic Brauer algebras in the following section. In Appendix A, we will derive a fourth description of the set, to demonstrate the connection with the recently developed arrow diagram calculus of \cite{gang}.

\subsection{Terminology and procedures on skew diagrams}\label{SecSkew}
In diagrams, we will use the terms `below', `above', `left of', and `right of' in the strict sense. A box~$b$ is thus above a box~$c$ if they are in the same column and~$b$ is in a row $i$ while $c$ is row $j$ with $j>i$. 

\subsubsection{Disjoint and connected diagrams}A skew Young diagram is connected if it is does not consist of two disjoint diagrams. Two diagrams are disjoint if there is no box of the first diagram which shares a side with a box of the second diagram.

\subsubsection{Addable and removable boxes}
An addable, resp. removable, box for a skew Young diagram is a box which can be added to, resp. removed from, the diagram such that the outcome is still a skew Young diagram.
\begin{itemize}
\item A d-addable box~$b$ of~$\kappa$ is an addable box  such that there are no boxes in~$\kappa$ to the right of or below $b$.
\item A d-removable box~$b\in\kappa$ is a removable such that there are no boxes in~$\kappa$ to the right of or below $b$.
\item A u-addable box~$b$ of~$\kappa$ is an addable box such that there are no boxes in~$\kappa$ to the left of or above $b$.
\item A u-removable box~$b\in\kappa$ is a removable box such that there are no boxes in~$\kappa$ to the left of or above $b$.
\end{itemize}

\subsubsection{Content}
From each interpretation as the difference of two Young diagrams, a skew Young diagram inherits (anti)content for its boxes. We will generally choose an arbitrary normalisation, by fixing the content of one box, which then determines the content of all other boxes. A box with content $p$ will be referred to as a~$p$-box. In Example~\ref{Ex1} we display a skew diagram with its contents for one normalisation. This skew diagram can be interpreted as the difference of the partition~$(5,5,5,3,1,1)$ and~$(3,2,2)$. The latter interpretation would of course lead to a different normalisation of the content.

\subsubsection{Hooks}
A hook is a skew diagram which has (for an arbitrary normalisation of its content) no two boxes with same content.
We will refer to the unique box in a hook with maximal, resp. minimal, content as the maximal, resp. minimal, box of the hook.
Unless specified otherwise, we will always assume hooks to be connected.

The {\em height} $\htt(\gamma)$ of a hook~$\gamma$ is the number of rows it has boxes in. The {\em width} $\wdd(\gamma)$ is the number of columns it has boxes in. The size (number of boxes) of a hook satisfies
\begin{equation}\label{sizehook}\mbox{size}(\gamma)\;=\;\htt(\gamma)+\wdd(\gamma)-1.
\end{equation}

\begin{ex}\label{Ex1}
The skew Young diagram
\abovedisplayskip-.5em
\belowdisplayskip-.5em
\[
\begin{minipage}{2.05cm} \young(:::89,::678,::567,234,1,0) \end{minipage}
\]
has 3 u-removable boxes, with content $2$, $6$ and~$8$, as well as 3 d-removable boxes, with content $0$, $4$ and~$7$.
In the following diagram we draw all d-addable, resp. u-addable boxes which would lead to connected diagrams, and label them by d, resp. u:
\abovedisplayskip-.5em
\belowdisplayskip-.5em
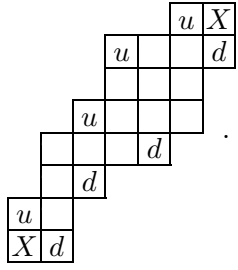
\[
\begin{minipage}{2.85cm} \young(:::::uX,:::u\leeg \leeg d,:::\leeg\leeg\leeg,::u\leeg\leeg\leeg,:\leeg\leeg\leeg d,:\leeg d,u\leeg,Xd) \end{minipage}.
\]
There are infinitely many more addable boxes, which are simultaneously u- and d-addable, they lead to non-connected diagrams when added on. Examples of these are the two boxes marked $X$.
\end{ex}

\begin{ddef}\label{DefProc}
Let $\kappa$ be an arbitrary skew diagram, with fixed content. For any $q\in\mZ$, we define skew diagrams~$\Pq(\kappa)$, $\oPq(\kappa)$, $\Eq(\kappa)$, $\oEq(\kappa)$ as follows.
\begin{enumerate}[(i)]
\item If $\kappa$ has a d-addable $q$-box~$b_1$ and, furthermore, $\kappa\cup\{b_1\}$ has a u-removable $q$-box~$b_2$, we set $\Pq(\kappa):=\kappa\cup\{b_1\}\backslash \{b_2\}$. In all other cases, we set $\Pq(\kappa)=\varnothing$.
\item If $\Pq(\kappa)$ allows no d-addable $q+1$-box, we set $\oPq(\kappa)=\Pq(\kappa)$, otherwise we set $\oPq(\kappa)=\varnothing$.
\item If $\kappa$ has a u-addable $q\minus 1$-box~$b_2$ and, furthermore, $\kappa\cup\{b_2\}$ has a d-addable $q$-box~$b_1$, we set $\Eq(\kappa):=\kappa\cup\{b_1,b_2\}$. In all other cases, we set $\Eq(\kappa)=\varnothing$.
\item If $\Eq(\kappa)$ allows no d-addable $q\,\minus\,1$-box, we set $\oEq(\kappa)=\Eq(\kappa)$, otherwise we set $\oEq(\kappa)=\varnothing$.
\end{enumerate}
\end{ddef}

\subsubsection{}
Loosely speaking, $\Pq(\kappa)$ is obtained from~$\kappa$ by `pushing down' all $q$-boxes one position along the diagonal, if that is possible. On the other hand, $\Eq(\kappa)$ is obtained from~$\kappa$ by `extending' $\kappa$ with a~$q\minus 1$-box on the upper rim and a~$q$-box on the lower rim, if that is possible.

\subsection{Iterative descriptions}
We define sets of skew Young diagrams by making use of the the procedures introduced in Definition~\ref{DefProc}. To apply these procedures we have to choose an arbitrary normalisation of the content in each step. The definition of the sets are not influenced by this as we will always consider the procedures for arbitrary $q\in\mZ$. In particular, we stress that the sets are to be considered as sets of skew Young diagrams which have no fixed normalisation of content or position in space.

\begin{ddef}\label{DefUp1} The set $\Upsilon$ of skew Young diagrams is determined by the following three properties:
\begin{itemize}
\item We have $\varnothing\in \Upsilon$.
\item If $\kappa\in \Upsilon$, we have $\Pq(\kappa)\in\Upsilon$, for any $q\in\mZ$.
\item If $\kappa\in \Upsilon$, we have $\Eq(\kappa)\in\Upsilon$, for any $q\in\mZ$.
\end{itemize}
\end{ddef}

\begin{ddef}\label{DefUp2} The set $\overline{\Upsilon}$ of skew Young diagrams is determined by the following three properties:
\begin{itemize}
\item We have $\varnothing\in \overline{\Upsilon}$.
\item If $\kappa\in \overline{\Upsilon}$, we have $\oPq(\kappa)\in\overline{\Upsilon}$, for any $q\in\mZ$.
\item If $\kappa\in \overline{\Upsilon}$, we have $\oEq(\kappa)\in\overline{\Upsilon}$, for any $q\in\mZ$.
\end{itemize}
\end{ddef}

Clearly we have $\overline{\Upsilon}\subseteq\Upsilon$.

\begin{ex}\label{Ex1Proc}
We have 
$\mathbf{E}_1(\varnothing)=\overline{\mathbf{E}}_1(\varnothing)= \begin{minipage}{.8cm} \scalebox{.8}{\young(01)} \end{minipage}$.
Furthermore
$$\mathbf{E}_3(\begin{minipage}{.75cm} \scalebox{.8}{\young(01)} \end{minipage})=\begin{minipage}{1.15cm} \scalebox{.8}{\young(:23,01)} \end{minipage},\quad \overline{\mathbf{E}}_3(\begin{minipage}{.75cm} \scalebox{.8}{\young(01)}\end{minipage})=\varnothing\quad\mbox{and}\qquad\mathbf{E}_{\minus 1}(\begin{minipage}{.75cm} \scalebox{.8}{\young(01)}\end{minipage})=\overline{\mathbf{E}}_{\minus 1}(\begin{minipage}{.75cm} \scalebox{.8}{\young(01)}\end{minipage})=\begin{minipage}{1.15cm} \scalebox{.8}{\young(:01,\mintwee \mineen)} \end{minipage}.$$
Applying $\Eq$ for other values of~$q$ yields either $\varnothing$ or disconnected skew diagrams consisting of two diagrams of shape {\tiny $\yng(2)$}.

We also find
$$\mathbf{P}_2(\begin{minipage}{1.1cm} \scalebox{.8}{\young(:23,01)} \end{minipage})=\overline{\mathbf{P}}_2(\begin{minipage}{1.15cm} \scalebox{.8}{\young(:23,01)} \end{minipage})=\begin{minipage}{1.1cm} \scalebox{.8}{\young(::3,012)} \end{minipage}.$$
Other d-addable boxes of {{$\begin{minipage}{1.15cm} \scalebox{.8}{\young(:23,01)} \end{minipage}$}} have either content bigger than 3 or lower than $0$, meaning there can never be a corresponding $u$-removable box, so $\Pq$ yields $\varnothing$ for $q\not=2$. We also find that $\Pq$ acting on $\begin{minipage}{1.15cm} \scalebox{.8}{\young(::3,012)} \end{minipage}$ yields $\varnothing$, for all $q\in\mZ$. On the other hand, we have
$$\mathbf{E}_5(\begin{minipage}{1.1cm} \scalebox{.8}{\young(::3,012)} \end{minipage})=\begin{minipage}{1.5cm} \scalebox{.8}{\young(::45,::3,012)} \end{minipage}\text{ and } \overline{\mathbf{E}}_{\minus 1}(\begin{minipage}{1.1cm} \scalebox{.8}{\young(::3,012)} \end{minipage})=\begin{minipage}{1.5cm} \scalebox{.8}{\young(:::3,:012,\mintwee\mineen)} \end{minipage}.$$
We also have
$$\mathbf{E}_5(\begin{minipage}{1.1cm} \scalebox{.8}{\young(:23,01)} \end{minipage})=\begin{minipage}{1.5cm} \scalebox{.8}{\young(::45,:23,01)} \end{minipage},\qquad \overline{\mathbf{E}}_5(\begin{minipage}{1.1cm} \scalebox{.8}{\young(:23,01)} \end{minipage})=\varnothing,$$
$$\mathbf{E}_1(\begin{minipage}{1.1cm} \scalebox{.8}{\young(:23,01)} \end{minipage})=\overline{\mathbf{E}}_1(\begin{minipage}{1.1cm} \scalebox{.8}{\young(:23,01)} \end{minipage})=\begin{minipage}{1.1cm} \scalebox{.8}{\young(123,012)} \end{minipage},\qquad \mathbf{E}_{\minus 1}(\begin{minipage}{1.1cm} \scalebox{.8}{\young(:23,01)} \end{minipage})=\overline{\mathbf{E}}_{\minus 1}(\begin{minipage}{1.1cm} \scalebox{.8}{\young(:23,01)} \end{minipage})=\begin{minipage}{1.5cm} \scalebox{.8}{\young(::23,:01,\mintwee \mineen)} \end{minipage},$$
while $\Eq$ for other values of~$q$ yields either $\varnothing$ or disconnected diagrams.
\end{ex}

\begin{ex}\label{Ex6}
The connected non-zero diagrams in~$\Upsilon$, or in~$\overline{\Upsilon}$ of size up to 6 are given by
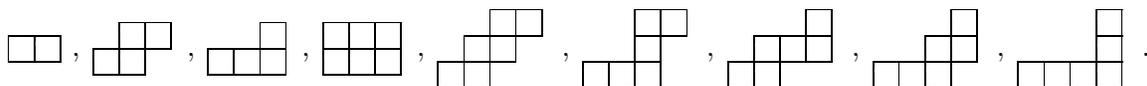
\[
\begin{minipage}{.85cm} \scalebox{.8}{\young(\leeg\leeg)} \end{minipage}, \;
\begin{minipage}{1.25cm} \scalebox{.8}{\young(:\leeg\leeg,\leeg\leeg)} \end{minipage}, \; 
\begin{minipage}{1.25cm} \scalebox{.8}{\young(::\leeg,\leeg\leeg\leeg)} \end{minipage}, \;
\begin{minipage}{1.25cm} \scalebox{.8}{\young(\leeg\leeg\leeg,\leeg\leeg\leeg)} \end{minipage}, \;
\begin{minipage}{1.65cm} \scalebox{.8}{\young(::\leeg\leeg,:\leeg\leeg,\leeg\leeg)} \end{minipage}, \;
\begin{minipage}{1.65cm} \scalebox{.8}{\young(::\leeg\leeg,::\leeg,\leeg\leeg\leeg)} \end{minipage}, \;
\begin{minipage}{1.65cm} \scalebox{.8}{\young(:::\leeg,:\leeg\leeg\leeg,\leeg\leeg)} \end{minipage}, \;
\begin{minipage}{1.65cm} \scalebox{.8}{\young(:::\leeg,::\leeg\leeg,\leeg\leeg\leeg)} \end{minipage}, \;
\begin{minipage}{1.65cm} \scalebox{.8}{\young(:::\leeg,:::\leeg,\leeg\leeg\leeg\leeg)} \end{minipage}.
\]
\end{ex}

\subsection{Description in terms of rim hooks}

\subsubsection{}\label{DefCk} Consider an arbitrary connected skew Young diagram $\kappa$.
Take the collection of all boxes $b\in\kappa$ such that~$b$ is the right-most box in $\kappa$ with content $\con(b)$. This collection forms a hook, which we denote by~$\kappa^{(0)}$. The diagram $\kappa\backslash \kappa^{(0)}$ is again a skew Young diagram, possibly disconnected. We construct the outer hooks of each connected component, as above, yielding a set of hooks $\{\kappa^{(1)},\ldots,\kappa^{(l)}\}$. We remove them from~$\kappa\backslash \kappa^{(0)}$ and proceed as above. Hence, we construct a set of rim hooks (where we consider them as fixed in space) which together form $\kappa$. This set will be denoted by~$C(\kappa)$. The set $C(\kappa)$ of a disconnected skew partition is just the union of the corresponding sets for its components.

\subsubsection{}\label{DefCov} One hook~$\nu^1$ is {\em nested} in another hook~$\nu^2$ if each lower and right side of a box in $\nu^1$ is shared with another box in $\nu^1$ or $\nu^2$. 
We define a {\em covering} of a skew partition~$\kappa$ to be a decomposition of~$\kappa$ into (connected) hooks, such that any two hooks are either disjoint or nested.

\subsubsection{} By construction, $C(\kappa)$ as defined in \ref{DefCk} is a covering of~$\kappa$ in the sense of \ref{DefCov}. Moreover, one immediately verifies that a covering is unique. Assume for instance that~$\kappa$ is connected and has some covering $C$. Take any box~$b$ in $\kappa^{(0)}$ (as defined in \ref{DefCk}). It must belong to some hook~$\gamma$ in $C$. The boxes in $\kappa^{(0)}$ with content $\con(b)\pm1$ cannot belong to hooks in $C$ which allow a nesting with a hook containing $b$. Thus it follows that~$\kappa^{(0)} \subset \gamma$, hence they are equal and $\kappa^{(0)}\in C$. One proceeds iteratively and obtains that~$C$ already coincides with $C(\kappa)$.

\begin{ddef}
Let $\Gamma_0$ be the set of all (connected) hooks $\gamma$ which satisfy the following two conditions:
\begin{itemize}
\item (HW-condition) We have $\wdd(\gamma)=\htt(\gamma)+1$;
\item (D-condition) The anticontent of the minimal box in $\gamma$ is the minimal value of the anticontents of the boxes in $\gamma$. Equivalently, no box in $\gamma$ lies strictly above the positive diagonal drawn from the minimal box.
\end{itemize}
Let $\Gamma$ be the set of skew diagrams~$\kappa$, where each hook in its covering $C(\kappa)$ belongs to $\Gamma_0$. \end{ddef}

Note that, by equation~\eqref{sizehook}, the HW-condition immediately implies that any hook in $\Gamma_0$ has an even number of boxes.

\begin{ex}
The unique hook with $2k$ boxes, for $k\in\mN$, such that there are never more than two boxes on the same row or column, and such that the minimal box is alone in its column, will be referred to as a {\em staircase}. The staircases of size $2$, $4$ and~$6$ are given by
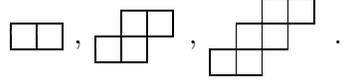
\[
\begin{minipage}{.85cm} \scalebox{.8}{\young(\leeg\leeg)} \end{minipage}, \; 
\begin{minipage}{1.25cm} \scalebox{.8}{\young(:\leeg\leeg,\leeg\leeg)} \end{minipage}, \;
\begin{minipage}{1.65cm} \scalebox{.8}{\young(::\leeg\leeg,:\leeg\leeg,\leeg\leeg)} \end{minipage}.
\]
All staircases belong to $\Gamma_0$. They correspond to those elements in $\Gamma_0$ where the D-condition is only scantly satisfied.
\end{ex}

We will require the following elementary property of the set $\Gamma$.
\begin{lemma}\label{LemGam}
Let $\kappa,\nu$ be skew diagrams, where $\kappa$ is obtained from~$\nu$ by adding two boxes {{$\begin{minipage}{.45cm} \scalebox{.8}{\yng(1,1)} \end{minipage}$}} in a column of~$\nu$ such that no boxes in $\nu$ are above or left from the added boxes. Then at most one of the skew diagrams~$\kappa,\nu$ is in $\Gamma$.
\end{lemma} 
\begin{proof}
We start by observing that any hook in $C(\nu)$ is contained in a hook in $C(\kappa)$. Indeed, there are three possibilities for the outer rim hook~$\kappa^{(0)}$. Either it is the same as $\nu^{(0)}$, it is $\nu^{(0)}$ together with the lower of the added boxes, or it is $\nu^{(0)}$ with both added boxes. In the third case, all other hooks in the coverings of~$\kappa$ and $\nu$ coincide. In the second case, one of the connected outer rim hooks in $\kappa\backslash \kappa^{(0)}$ will consist of the second added box together and an outer rim hook of~$\nu\backslash\nu^{(0)}$, while al other elements $C(\kappa)$ and $C(\nu)$ will be identical. In the first case, one removes the rim hook~$\kappa^{(0)}=\nu^{(0)}$ and applies the above procedure to $\kappa\backslash \kappa^{(0)}$ and $\nu\backslash \nu^{(0)}$.

Assume first that the two added boxes belong to different elements of~$C(\kappa)$, which implies that either $C(\nu)$ or $C(\kappa)$ contains hooks of odd sizes. As all elements of~$\Gamma_0$ must be of even size, $\kappa$ and~$\nu$ cannot both be in $\Gamma$.

Now we assume that the two added boxes belong to the same element $\gamma\in C(\kappa)$. Hence, $\gamma$ is obtained from some $\delta\in C(\nu)$ by adding two boxes in the same column such that no boxes in $\delta$ are above or left of the added boxes. This means that the two boxes are added to $\delta$ in such a way that either 
\begin{enumerate}[(a)]
\item one of them is the minimal box in $\gamma$; 
\item one of them is the maximal box in $\gamma$, with the maximal box in $\delta$ below the added boxes.
\end{enumerate}
In case (a), the D-condition in $\gamma$ is clearly violated, so $\kappa\not\in \Gamma$. In case (b) the HW-condition of either $\gamma$ or $\delta$ must be violated, so either $\kappa\not\in \Gamma$ or $\nu\not\in\Gamma$.
\end{proof}

\subsection{Equivalence of the three descriptions}

\begin{thm}\label{ThmEqui}
We have $\overline\Upsilon=\Upsilon=\Gamma$.
\end{thm}
This theorem follows from the subsequent Propositions~\ref{Prop1} and \ref{Prop2}, and the obvious inclusion~$\overline{\Upsilon}\subseteq\Upsilon$.

\begin{lemma}
If $\nu\in\Gamma$, then both $\Eq(\nu)$ and~$\Pq(\nu)$ are in~$\Gamma$, for all $q\in\mZ$.
\end{lemma}
\begin{proof}
Set $\kappa=\Pq(\nu)$ and restrict to the non-trivial case $\kappa\not=\varnothing$. Denote by~$b_1,\ldots,b_r$ the $q$-boxes of~$\nu$, ordered from top left to bottom right.
 Thus $\Pq$ will delete the box~$b_1$ and add a new box~$b_{r+1}$ on the $q$-diagonal below $b_r$. Note that since $b_{r+1}$ is an addable box, $\nu$ contains a~$q\,\minus\,1$-box~$a_r$ directly below $b_r$ and a~$q+1$-box~$c_r$ directly to the right of~$b_r$. Since $\nu$ is a skew shape, the same holds for boxes $a_i$ and~$c_i$ next to the box~$b_i$ for $1 \leq i < r$. As $\nu \in \Gamma$, it can be covered by nested hooks. Thus there are hooks $\gamma_1,\ldots,\gamma_r$ such that the box~$b_i$ is contained in~$\gamma_i$ for all $i$. Since these hooks only share faces with hooks in which they are either nested or that are nested inside them, this implies that also $a_i$ and~$c_i$ are contained in~$\gamma_i$ for all $i$. This is evident for $a_r$ and~$c_r$ and then follows successively for all others. Thus we can easily cover $\kappa$ by hooks by leaving all hooks except $\gamma_1,\ldots,\gamma_r$ unchanged and for each $\gamma_i$, we delete the box~$b_i$ and add the box~$b_{i+1}$. As we do not touch the minimal and maximal boxes of $\gamma_i$ the HW-condition remains satisfied and as we only push boxes down the diagonal, also the D-condition remains satisfied. So the covering of~$\kappa$ is by hooks of~$\Gamma_0$ and thus proves the claim.

Now set $\kappa=\Eq(\nu)$ and assume again $\kappa\not=\varnothing$. As in the previous situation we use labels $a_1,\ldots,a_r$ for the boxes on the $q\,\minus\,1$-diagonal of~$\nu$ and~$b_1,\ldots,b_r$ for those on the $q$-diagonal. That the number of boxes is the same is due to the fact that procedure $\Eq$ implies that there is a u-addable box~$a_0$ and a d-addable box~$b_{r+1}$. So we also find that every $a_i$ is directly below $b_i$, for $1\le i\le r$. As $b_{r+1}$ is d-addable to $\nu$, this also implies that there is a~$q+1$-box~$c_r$ to the right of~$b_r$, which in turn implies that such a box~$c_i$ exists for to the right of every $b_i$ for $1 \leq i \leq r$. As before denote by~$\gamma_1,\ldots,\gamma_r$ the hooks in the covering of~$\nu$ such that~$b_i$ is contained in~$\gamma_i$. The same argument as above gives that also $a_i$ and~$c_i$ are contained in~$\gamma_i$. We can thus modify the hooks in the same way and delete $b_i$ from~$\gamma_i$ and add $b_{i+1}$ to $\gamma_i$. In contrast to the previous situation this leaves $b_1$ and~$a_0$ as the only boxes in~$\kappa$ not contained in a hook. If the box directly above $b_1$ is not contained in~$\kappa$ we just add these two as a 2-hook $\begin{minipage}{.75cm} \scalebox{.8}{\young(\leeg\leeg)} \end{minipage}$, nested inside $\gamma_1$ by construction. If the box directly above $b_1$ is contained in~$\kappa$, it is also contained in a hook~$\gamma_0$ nested inside $\gamma_1$. In this case we add both boxes to this hook  $\gamma_0$ which still satisfies both the HW-condition and the D-condition on the hooks. Thus the claim is also proved in this case.
\end{proof}
This lemma implies immediately the following statement.
\begin{prop}\label{Prop1}
We have $\Upsilon \subseteq \Gamma$.
\end{prop}

Now we start the proof of the inclusion~$\Gamma\subseteq\overline\Upsilon$

\begin{lemma}\label{Lemdq1}
Take~$\delta\in \Gamma_0$ nested in $\gamma\in\Gamma_0$. Assume that~$\gamma$ has a d-removable $q$-box, but no d-addable $q+1$-box, and that~$\delta$ contains a~$q$-box. Then the $q$-box in $\delta$ is d-removable and either
\begin{enumerate}[(i)]
\item $\delta$ allows no d-addable $q+1$-box and contains the shape $\begin{minipage}{.75cm} \scalebox{.8}{\young(:\qpluseen,\qmineen \qss)} \end{minipage}$; or
\item $\delta$ contains no $q+1$-box, but contains the shape $\begin{minipage}{.75cm} \scalebox{.8}{\young(\qmineen \qss)} \end{minipage}$.
\end{enumerate}
\end{lemma}
\begin{proof}
As $\gamma$ has a d-removable $q$-box, but no d-addable $q+1$-box, $\gamma$ must contain a~$q+1$-box above its $q$-box, but no $q+2$-box right of the $q+1$-box. The D-condition then implies that the $q$-box cannot be minimal, so there must also be a~$q\,\minus\,1$-box left of the $q$-box.
Hence $\gamma$ contains the shape
\[
\begin{minipage}{.8cm} \scalebox{.8}{\young(:\qpluseen,\qmineen \qss)} \end{minipage},
\]
without a~$q+2$-box to the right of the $q+1$-box. Any $q$-box in~$\delta$ is thus clearly d-removable.

Assume first that~$\delta$ contains a~$q$-box, but allows no d-addable $q+1$-box. Just like we did for $\gamma$ we can show that~$\delta$ contains $\begin{minipage}{.75cm} \scalebox{.8}{\young(:\qpluseen,\qmineen \qss)} \end{minipage}$, which means we are in situation (i).

 Now assume that~$\delta$ contains a~$q$-box, and also has a d-addable $q+1$-box. We prove that the two conditions in (ii) are satisfied.
If $\delta$ would contain a~$q+1$-box above its $q$-box, then it needs to contain a~$q+2$-box right of this $q+1$-box in order to allow a~$d$-addable $q+1$-box. However, as $\delta$ is nested in~$\gamma$, this would require $\gamma$ to have a~$q+2$-box right of its $q+1$-box, which is not the case, a contradiction. In particular we find that the $q$-box in~$\delta$ is the maximal box. As $\delta$ cannot just be one box, there must be a~$q\,\minus\,1$-box. As the $q\,\minus\,1$-box below the $q$-box in~$\delta$ already belongs to $\gamma$, this must be the $q\,\minus\,1$-box left of the $q$-box.
\end{proof}

\begin{prop}\label{Propplus}
If $\kappa\in\Gamma$ has a d-removable $q$-box but no d-addable $q+1$-box, then one of the following is true
\begin{enumerate}[(i)]
\item $\kappa=\oPq(\tilde\kappa)$ for some $ \tilde\kappa\in\Gamma$;
\item $\kappa=\oEq(\tilde\kappa)$ for some $ \tilde\kappa\in\Gamma$;
\item $\kappa$ has a d-addable $q\,\minus\, 1$-box and the highest $q$-box in~$\kappa$ is maximal in its hook in~$C(\kappa)$.
\end{enumerate}
\end{prop}
\begin{proof}
Let $\gamma^0$ be the hook in~$C(\kappa)$ containing the the d-removable $q$-box, $\gamma^1$ the hook containing the $q$-box in the column left of the d-removable one, until we reach $\gamma^k$ containing the left-most $q$-box in~$\kappa$.

(a) Assume first that for each $\gamma^j$, its $q$-box is d-removable and that each $\gamma^j$ has no d-addable $q+1$-box. By Lemma~\ref{Lemdq1}, each $q$-box in~$\kappa$ has a~$q+1$-box above it and a~$q\minus 1$ box to its left. We thus easily find that~$\kappa=\oPq(\tilde\kappa)$ for some skew diagram $\tilde\kappa$ with a covering by hooks, which are either hooks of~$\kappa$, or of the form $\tilde{\gamma}^j$ with $\gamma^j=\Pq(\tilde{\gamma}^j)$. If $\tilde\kappa$ would not be in~$\Gamma$, there should be a~$\tilde{\gamma}^j$ which does not satisfy the D-condition. This would imply in particular that~$\gamma^j$ contains a~$q+2$-box right of its $q+1$-box, contradicting the assumption that~$\gamma^j$ has no d-addable $q+1$-box. This means we are in situation (i).

(b) Now assume that the assumption in (a) is not satisfied. Lemma~\ref{Lemdq1} implies that there is a~$\gamma^j$ which contains no $q+1$-box. Note that such a~$\gamma^j$ cannot have a hook with $q$-box nested within, hence $j=k$. Lemma~\ref{Lemdq1} implies further that~$\gamma^k$ contains a~$q\,\minus\,1$-box left of the $q$-box. By the D-condition, there is no $q\,\minus\,2$-box to the left of this $q\,\minus\,1$-box. It thus follows that~$\kappa=\Eq(\tilde\kappa)$ for some skew diagram $\tilde\kappa$. As above it follows that for $j<k$, we have $\gamma^j=\oPq(\tilde\gamma^j)$ for $\Gamma_0$ hooks $\tilde{\gamma}^j\in C(\tilde\kappa)$. Furthermore, there is a hook~$\tilde{\gamma}^k\in C(\tilde\kappa)$ with $\gamma^k=\Eq(\tilde\gamma^k)$ which by construction is in~$\Gamma_0$. As $C(\tilde\kappa)$ consists of the $\tilde\gamma^j$, for $0\le j\le k$, along with some elements of~$C(\kappa)$, we find $\tilde\kappa\in\Gamma$. 

Now assume that~$\kappa$ allows no d-addable $q\,\minus\,1$-box, by definition we then have $\kappa=\oEq(\tilde\kappa)$ and we are in situation (ii). If $\kappa$ allows a d-addable $q\,\minus\,1$-box, we are in situation (iii).
\end{proof}

\begin{prop}\label{Prop2}
We have $\Gamma\subseteq\overline\Upsilon$.
\end{prop}
\begin{proof}
It suffices to prove that for any $\kappa\in\Gamma$, we can find $\tilde\kappa\in\Gamma$, such that~$\kappa=\oPq(\tilde\kappa)$ or $\kappa=\oEq(\tilde\kappa)$ for some $q\in\mZ$. Indeed, when we iterate this it becomes clear that the operators $\oEq$ can only be used a finite a number of time, as $\kappa$ only contains a finite number of boxes. On the other hand, the operators $\oPq$ leave the left-most and the right-most boxes of~$\kappa$ invariant, meaning that they can also only be applied a finite number of times. 

Now consider arbitrary $\kappa\in\Gamma$. Proposition~\ref{Propplus} already provides $\tilde\kappa$ except in the following situations:
\begin{enumerate}[(a)]
\item for every $q\in \mZ$ for which $\kappa$ contains a d-removable $q$-box, $\kappa$ allows a d-addable $q+1$-box;
\item for every $q\in \mZ$ for which $\kappa$ contains a d-removable $q$-box but no d-addable $q+1$-box, $\kappa$ has a d-addable $q\,\minus\, 1$-box and the highest $q$-box in~$\kappa$ is maximal in its hook in~$C(\kappa)$.
\end{enumerate}
In case (a), $\kappa$ must be a staircase, which can be obtained from a smaller staircase by applying~$\overline{\mathbf{E}}$. Assume therefore that~$\kappa$ is as in (b). Take $q$ the minimal value for which $\kappa$ has a d-removable $q$-box but no d-addable $q+1$-box. Then take the maximal $r<q$ such that $\kappa$ has a d-removable $r$-box but no d-addable $r\,\minus\,1$-box. Note that~$r$ must exist by the HW-condition of elements in~$\Gamma_0$. The rim hook~$\gamma^0$ in~$C(\kappa)$ which contains the d-removable $q$-box must contain the shape

\begin{equation}\label{eqqr}\begin{minipage}{1.25cm} \scalebox{.8}{\young(:::::::\qpluseen,::::::\qmineen \qss,:::::\qmindrie \qmintwee,::::\cdots\qminvier,::::\cdots,::\rpluseen,\rmintwee\rmineen \rss)} \end{minipage}\end{equation}
where there is no $q+2$-box right of the $q+1$-box. If $\kappa$ contains no other $q$-boxes, $\kappa$ is clearly of the form $\oPq(\tilde\kappa)$ for some $\tilde\kappa\in\Gamma$, so we assume existence of more $q$-boxes. Define $\gamma^1\in C(\kappa)$ as the hook containing the $q$-box next to the one displayed. As $\gamma^1$ is in~$\Gamma_0$ and must be nested in~$\gamma^0$ we find it must contain the $r\,\minus\,1$-box immediately above the displayed $r\,\minus\,2$-box.

(I) If there are no further $q$-boxes in~$\kappa$ then assumption in (b) implies that the $q$-box in~$\gamma^1$ is the maximal one. The HW-condition on~$\gamma^1$ then implies that there is no box to the left of the $r\,\minus\,1$-box in~$\gamma^1$. Since the $r\,\minus\,2$-box below the $r\,\minus\,1$-box in~$\gamma^1$ already belongs to $\gamma^0$, it follows that~$\gamma^1$ is a staircase starting at its $r\,\minus\,1$-box and ending at its $q$-box. In particular, there can be no $r\,\minus\,2$-box in~$\kappa$ next to the $r\,\minus\,1$-box in~$\gamma^1$. It then follows that~$\kappa=\overline{\mathbf{E}}_r(\tilde\kappa)$ for some skew diagram $\tilde\kappa$. 
Note that by construction there is a covering of~$\tilde\kappa$ by hooks which consists of hooks which are already in~$\kappa$, except for $\tilde{\gamma}^1$, which is obtained from~$\gamma^1$ by removing its $r\,\minus\,1$ and~$r$-box; and~$\tilde{\gamma}^0$, which is obtained from~$\gamma^0$ by pushing upwards its $r$-box. As the latter does not break the D-condition we find $\tilde\kappa\in\Gamma$.

(II) If there is another $q$-box in $\kappa$, not yet in $\gamma^0$ or $\gamma^1$, we consider the hook~$\gamma^2\in C(\kappa)$ containing it. As $\gamma^2$ is nested in $\gamma^1$, it follows that~$\gamma^1$ contains a~$q+1$-box above its $q$-box. The HW-condition on~$\gamma^1$ then implies that its $r\,\minus\,1$-box cannot be its minimal box and thus there is a~$r\,\minus\,2$-box left of the $r\,\minus\,1$-box. All of the above then implies that~$\gamma^1$ also contains a shape \eqref{eqqr}. Furthermore, if there would be a~$q+2$-box in $\gamma^1$ right of its $q+1$-box, this would contradict its nesting inside $\gamma^0$ as we already know that~$\gamma^0$ has no $q+2$-box right of its $q+1$-box. In conclusion, the hook~$\gamma^1$ satisfies all the properties of~$\gamma^0$ that we have used above.

We can thus proceed iteratively and apply the procedure in (I) in case $\gamma^2$ contains the highest $q$-box, or procedure (II) in case there are more $q$-boxes. In conclusion, if there are $k$ $q$-boxes in~$\kappa$, we find that~$\gamma^j$ for $0\le j<k$ contain a shape \eqref{eqqr}, 
 while~$\gamma^k$ must be a staircase, and there exists~$\tilde\kappa\in\Gamma$ for which $\kappa=\overline{\mathbf{E}}_r(\tilde\kappa)$.
\end{proof}

\section{Cell multiplicities}\label{SecCell}
In this section we determine the cell multiplicities of the periplectic Brauer algebra completely. We will freely use Theorem~\ref{ThmEqui} and hence always apply the definition of~$\Gamma=\Upsilon=\overline{\Upsilon}$ which is appropriate to the situation. 

\subsection{Vanishing results}
\begin{lemma}\label{CorRes}
Consider $\lambda\in\LL_r$, $\mu\vdash r$ and assume that~$[W_r(\lambda):L_r(\mu)]\not=0$. If for $q\in\mZ$, we have $\widetilde\mu\in\RR(\mu)_q$, then
there exists~$\tilde\lambda\in\RR(\lambda)_q\sqcup\AAAA(\lambda)_q$,
for which $[W_{r\minus 1}(\tilde\lambda):L_{r\minus 1}(\tilde\mu)]\not=0$.
\end{lemma}

\begin{proof}
By equation~\eqref{eqSC}, we have $[W_r(\lambda)_q:L_{r\minus 1}(\tilde\mu)]\not=0$.
The result thus follows from Proposition~\ref{Propalpha}.\end{proof}

\begin{prop}\label{rrm}
Assume that~$[W_r(\lambda):L_r(\mu)]\not=0$, then $\lambda\subseteq \mu$ and~$ \mu/\lambda\in \Upsilon$.
\end{prop}
\begin{proof}
The condition~$\lambda\subseteq\mu$ is Lemma~\ref{2Lem}(i), so we only prove $ \mu/\lambda\in \Upsilon$.
For $r\le 5$, this follows from \cite[Section~9]{PB1} and Example~\ref{Ex6}, so we proceed by induction on~$r$. By Lemma~\ref{2Lem}(ii), we can restrict to the case $\mu\vdash r$. Now assume that, for $\lambda\subset \mu\vdash r$, we have $[W_r(\lambda):L_r(\mu)]\not=0$. Consider an arbitrary pair of partitions $\tilde\lambda,\tilde\mu$ as in Lemma~\ref{CorRes}. Using the induction hypothesis, we find that there is $\kappa\in \Upsilon$, with $\tilde\mu/ \tilde\lambda=\kappa$.

Firstly assume first that~$\tilde\lambda\in\RR(\lambda)_q$. Then we have $\mu=\tilde\mu\cup\{b_1\}$ and~$\lambda=\tilde\lambda\cup\{b_2\}$, for $q$-boxes $b_1,b_2$. Consequently, $\tilde\lambda\subset\mu$, so $\kappa\cup\{b_1\}=\mu/ \tilde\lambda$ is a skew Young diagram. Moreover, as the box~$b_1$ is addable to $\tilde\mu$, there is nothing in~$\tilde\mu$ to the right or below $b_1$. This implies in particular that there is nothing in~$\tilde\mu/\tilde\lambda=\kappa$ to the right of or below $b_1$, so $b_1$ is d-addable to $\kappa$.
As $\lambda\subset\mu$, also $\mu/\lambda =(\kappa\cup\{b_1\})\backslash\{b_2\}$ is a skew partition. As $b_2$ is addable to $\tilde\lambda$ it follows that nothing in~$\mu/ \tilde\lambda=\kappa\cup\{b_1\}$ is above or left of~$b_2$, so $b_2$ is u-removable from~$\kappa\cup\{b_1\}$. In conclusion, $\mu/\lambda=\Pq(\kappa)$, meaning that~$\mu/\lambda\in \Upsilon$.

Secondly assume that~$\tilde\lambda\in\AAAA(\lambda)_q$. Then we have $\mu=\tilde\mu\cup\{b_1\}$ and~$\tilde\lambda=\lambda\cup\{b_2\}$, for a~$q$-box~$b_1$ and a~$q\minus 1$-box~$b_2$. Hence $\lambda\subset\tilde\mu$, so $\kappa\cup \{b_2\}=\tilde\mu/\lambda$ is a skew diagram. If there would be a box in~$\kappa=\tilde\mu/\tilde\lambda$ above or to the left of~$b_2$, it could not be addable to $\tilde\lambda$, a contradiction. Hence $b_2$ is u-addable to $\kappa$. We also have $\kappa\cup\{b_1,b_2\}=\mu/\lambda$. As $b_1$ is addable to $\tilde\mu$ it is clearly $d$-addable to $\tilde\mu/ \lambda=\kappa\cup\{b_1,b_2\}$. Hence, $\mu/\lambda=\Eq(\kappa)$.

In both cases, Definition~\ref{DefUp1}, shows that~$\mu/\lambda\in\Upsilon$.
\end{proof}

\begin{cor}\label{CorGam}
Consider some partition~$\eta$ and~$q\in\mZ$, with $\lambda^1\in\RR(\eta)_q$ and~$\lambda^2\in\AAAA(\eta)_q$. For every partition~$\mu$, we then have the following chains of conclusions
\begin{enumerate}[(i)]
\item $[W(\lambda^1):L(\mu)]\not=0\;\;\Rightarrow\;\; \mu/\lambda^1\in\Upsilon\; \;\Rightarrow\; \; \mu/\lambda^2\not\in\Upsilon \; \;\Rightarrow\; \;[W(\lambda^2):L(\mu)]=0$;
\item $[W(\lambda^2):L(\mu)]\not=0\;\;\Rightarrow\;\; \mu/\lambda^2\in\Upsilon\; \;\Rightarrow\; \; \mu/\lambda^1\not\in\Upsilon \; \;\Rightarrow\; \;[W(\lambda^1):L(\mu)]=0$.
\end{enumerate}
\end{cor}
\begin{proof}
The diagram of~$\lambda^2$ is obtained from~$\lambda^1$ by adding a rim 2-hook {$\begin{minipage}{.45cm} \scalebox{.8}{\young(\qss,\qmineen)} \end{minipage}$}. Lemma~\ref{LemGam} thus implies that either
$\mu/\lambda^1\not\in\Upsilon$ or $\mu/\lambda^2\not\in\Upsilon$. The conclusion then follows from Proposition~\ref{rrm}.
\end{proof}

Using the above corollary, we can now find a stronger version of Lemma~\ref{CorRes}.

\begin{prop}\label{PropImpr}
Consider $\lambda\in\LL_r$, $\mu\vdash r$ and assume that~$[W_r(\lambda):L_r(\mu)]=k>0$. If for $q\in\mZ$, we have $\widetilde\mu\in\RR(\mu)_q$, then
precisely one of the following is true:
\begin{enumerate}[(i)]
\item there exists~$\tilde\lambda\in\RR(\lambda)_q$ with $[W_{r\minus 1}(\tilde\lambda):L_{r\minus 1}(\tilde\mu)]\ge k$, and for any $\eta\in\AAAA(\lambda)_q$ we have $\tilde\mu/\eta\not\in\Upsilon$,
\item there exists~$\tilde\lambda\in\AAAA(\lambda)_q$ with $[W_{r\minus 1}(\tilde\lambda):L_{r\minus 1}(\tilde\mu)]\ge k$, and for any $\eta\in\RR(\lambda)_q$ we have $\tilde\mu/\eta\not\in \Upsilon$.
\end{enumerate}
\end{prop}
\begin{proof}
We will assume that both $\lambda^1\in\RR(\lambda)_q$ and~$\lambda^2\in\AAAA(\lambda)_q$ exist, the other cases are easier to deal with.
Proposition~\ref{Propalpha} and equation~\eqref{eqSC} show that
$$[W_r(\lambda):L_r(\mu)]\;\le \; [W_{r\minus 1}(\lambda^1):L_{r\minus 1}(\tilde\mu)]+[W_{r\minus 1}(\lambda^2):L_{r\minus 1}(\tilde\mu)].$$
The conclusion thus follows from Corollary~\ref{CorGam}.
\end{proof}

\subsection{The cell multiplicities}
\begin{thm}\label{MainThm}
For any $\lambda\in \LL_r$ and~$\mu\in\Lambda_r$, we have
$$[W_r(\lambda):L_r(\mu)]\;=\;\begin{cases}1&\mbox{if}\;\, \lambda\subseteq\mu\;\mbox{ and }\; \mu/\lambda \in\Upsilon,\\0&\mbox{otherwise.}\end{cases}$$
\end{thm}

We start the proof with the following two lemmata.

\begin{lemma}\label{Lemprim1}
Consider $q\in\mZ$, $\tilde\lambda\subset\tilde\mu\vdash r\minus 1$ with $\tilde\mu/\tilde\lambda\in\Upsilon$, and~$\lambda\in\AAAA(\tilde\lambda)_{q+1}$, $\mu\in\AAAA(\tilde\mu)_{q+1}$, such that~$\mu/\lambda=\oPq(\tilde\mu/\tilde\lambda)$. Then we have
$$[W_{r\minus 1}(\tilde\lambda):L_{r\minus 1}(\tilde\mu)]\quad =\quad [W_{r}(\lambda):L_{r}(\mu)].$$
\end{lemma}
\begin{proof}
By assumption in Definition~\ref{DefProc}(ii), $\mu$ does not allow an addable $q+1$-box. As $\tilde\mu\in\RR(\mu)_q$, Lemma~\ref{2Lem}(ii) and Lemma~\ref{LemResS1} thus imply that
$$[W_{r\minus 1}(\tilde\lambda):L_{r\minus 1}(\tilde\mu)]=[W_{r+ 1}(\tilde\lambda):L_{r+ 1}(\tilde\mu)]\le [W_{r+ 1}(\tilde\lambda)_{q+1}:L_{r}(\mu)].$$
Proposition~\ref{Propalpha} implies
$$[W_{r+ 1}(\tilde\lambda)_{q+1}:L_{r}(\mu)]=[W_r(\lambda):L_{r}(\mu)]\,+\,\sum_{\nu\in\RR(\tilde\lambda)_{q+1}}[W_r(\nu):L_{r}(\mu)].$$
Since we assume $\mu/\lambda\in\Upsilon$, part of the chain in Corollary~\ref{CorGam}(ii) implies that the right-hand term vanishes.
The two displayed equalities above thus finally yield
$$[W_{r\minus 1}(\tilde\lambda):L_{r\minus 1}(\tilde\mu)]\le[W_r(\lambda):L_{r}(\mu)].$$

In order to prove the weak inequality in the other direction we can of course restrict to the assumption~$[W_r(\lambda):L_{r}(\mu)]\not=0$.
Proposition~\ref{PropImpr} then implies that 
$$[W_r(\lambda):L_r(\mu)]\;\le\; [W_{r\minus 1}(\tilde\lambda):L_{r\minus 1}(\tilde\mu)],$$
which concludes the proof.
\end{proof}

\begin{lemma}\label{Lemprim2}
Consider $q\in\mZ$, $\tilde\lambda\subset\tilde\mu\vdash r\minus 1$ with $\tilde\mu/\tilde\lambda\in\Upsilon$, and~$\lambda\in\RR(\tilde\lambda)_{q-1}$, $\mu\in\AAAA(\tilde\mu)_{q+1}$, such that~$\mu/\lambda=\oEq(\tilde\mu/\tilde\lambda)$. Then we have
$$[W_{r\minus 1}(\tilde\lambda):L_{r\minus 1}(\tilde\mu)]\quad= \quad [W_{r}(\lambda):L_{r}(\mu)].$$
\end{lemma}
\begin{proof}
By Definition~\ref{DefProc}(iv), the skew diagram $\mu/\lambda$ allows no d-addable $q\,\minus\,1$-box. Now $\tilde\mu/\lambda$ is obtained from~$\mu/\lambda$ by removing a d-removable $q$-box. The fact that~$\mu/\lambda$ allows no d-addable $q\,\minus\,1$-box implies that~$\tilde\mu/\lambda$ has a d-removable $q\,\minus\,1$-box. Consequently, the partition~$\tilde\mu$ contains a removable $q\,\minus\,1$-box. Lemma~\ref{2Lem}(ii) and Lemma~\ref{LemResS2} thus imply that
$$[W_{r\minus 1}(\tilde\lambda):L_{r\minus 1}(\tilde\mu)]=[W_{r+ 1}(\tilde\lambda):L_{r+ 1}(\tilde\mu)]\le [W_{r+ 1}(\tilde\lambda)_{q\minus 1}:L_{r}(\mu)].$$
Proposition~\ref{Propalpha} implies that
$$[W_{r+ 1}(\tilde\lambda)_{q\minus 1}:L_{r}(\mu)]\;=\; [W_{r}(\lambda):L_{r}(\mu)]\,+\, \sum_{\nu\in\AAAA(\tilde\lambda)_{q-1}}[W_{r}(\nu):L_{r}(\mu)]$$
As we assume $\mu/\lambda\in\Upsilon$, the right-hand term vanishes by part of the chain in Corollary~\ref{CorGam}(i). the above two displayed equations thus imply that 
$$[W_{r\minus 1}(\tilde\lambda):L_{r\minus 1}(\tilde\mu)]\;\le\;[W_r(\lambda):L_{r}(\mu)].$$
 The inequality
$$[W_r(\lambda):L_r(\mu)]\;\le\; [W_{r\minus 1}(\tilde\lambda):L_{r\minus 1}(\tilde\mu)]$$
follows from Proposition~\ref{PropImpr}, which concludes the proof.
\end{proof}

\begin{proof}[Proof of Theorem~\ref{MainThm}]
The vanishing when $\mu/\lambda\not\in\Upsilon$ is guaranteed by Proposition~\ref{rrm}, the statement $[W_r(\lambda):L_r(\mu)]=1$ for $\mu/\lambda\in\Upsilon$ can be reduced to the case $\mu\vdash r$, by Lemma~\ref{2Lem}(ii). For $r\le 5$ it follows from \cite[Section~9]{PB1} and Example~\ref{Ex6}, so we proceed by induction on $r$.

Assume $\mu\vdash r$ and $\mu/\lambda\in\Upsilon$. By Definition~\ref{DefUp2}, we must have $\mu/\lambda=\oPq(\kappa)$ or $\mu/\lambda=\oEq(\kappa)$ for some $\kappa\in\Upsilon$. There always exist partitions $\tilde\mu$ and $\tilde\lambda$ as Lemma~\ref{Lemprim1} or Lemma~\ref{Lemprim2}, with $\tilde\mu/\tilde\lambda=\kappa$.
Applying those lemmata hence yields the induction.
\end{proof}

\appendix

\section{Description of~$\Gamma$ in terms of arrow diagrams}
\subsection{Arrow diagrams}\label{SecArc}
\subsubsection{Weight diagrams}
Following \cite[Section~5]{BrMult}, to each partition~$\lambda$ we associate an infinite (strictly decreasing) sequence of integers $x_\lambda$ defined as
$$x_\lambda=(\lambda_1,\lambda_2\,\minus\,1,\lambda_3\,\minus\,2,\lambda_4\,\minus\,3,\ldots).$$
In analogy with \cite[Section~5.1]{gang}, the weight diagram of~$\lambda$ is then given by associating to each integer $i$ on the real line a white dot if $i\not\in x_\lambda$ and a black dot if $i\in x_\lambda$.

\begin{ex}\label{ExDiag}${}$
\begin{enumerate}[(i)]
\item For $\lambda=(1)$, the weight diagram $x_\lambda$ is given by 
$$ \xymatrix{  & \cdots&\underset{-4}{\bullet}  &\underset{-3}{\bullet} &\underset{-2}{\bullet}  &\underset{-1}{\bullet}  &\underset{0}{\circ}  &\underset{1}{\bullet}  &\underset{2}{\circ} &\underset{3}{\circ}  &\underset{4}{\circ} &\cdots} $$
\item For $\lambda=(3)$, the weight diagram $x_\lambda$ is given by 
$$ \xymatrix{  & \cdots &\underset{-4}{\bullet}  &\underset{-3}{\bullet} &\underset{-2}{\bullet}  &\underset{-1}{\bullet}  &\underset{0}{\circ}  &\underset{1}{\circ}  &\underset{2}{\circ} &\underset{3}{\bullet}  &\underset{4}{\circ} &\cdots} $$
\item For $\lambda=(2,1)$, the weight diagram $x_\lambda$ is given by 
$$ \xymatrix{  & \cdots &\underset{-4}{\bullet}  &\underset{-3}{\bullet} &\underset{-2}{\bullet}  &\underset{-1}{\circ}  &\underset{0}{\bullet}  &\underset{1}{\circ}  &\underset{2}{\bullet} &\underset{3}{\circ}  &\underset{4}{\circ}&\cdots } $$
\item For $\lambda=(3,2)$, the weight diagram $x_\lambda$ is given by 
$$ \xymatrix{  & \cdots &\underset{-4}{\bullet}  &\underset{-3}{\bullet} &\underset{-2}{\bullet}  &\underset{-1}{\circ}  &\underset{0}{\circ}  &\underset{1}{\bullet}  &\underset{2}{\circ} &\underset{3}{\bullet}  &\underset{4}{\circ} &\cdots} $$
\end{enumerate}
\end{ex}

\subsubsection{}\label{xtrans}It is clear that~$x_{\lambda'}$ is obtained from~$x_\lambda$ by reflecting the diagram with respect to a vertical line in the middle of~$0$ and~$1$, followed by changing the colours of all dots.

Any assignment of black and white boxes to $\mZ$ can be interpreted as the weight diagram of a (uniquely determined) partition if and only if there is a position~$i\in\mZ$ such that all dots to its left are black, and a position~$j\in\mZ$ such that all dots to its right are white.

\subsubsection{wb pairs and arrow pairs}
A {\em wb pair of dots} in $x_\lambda$ is a white dot at position~$i\in\mZ$ and a block dot at position~$j\in\mZ$, with $i<j$.
Such a pair is an {\em arrow pair of dots} in $x_\lambda$ if we further have that
\begin{itemize}
\item (hw-condition) the collection of dots in the interval $[i,j]$ contains precisely one more white dot than black dots;
\item (d-condition) one cannot draw a line to the left of~$j$ such that stictly between the line and position~$i$ one has fewer white than black dots.
\end{itemize}

\begin{ex} Consider the weight diagrams in Example~\ref{ExDiag}.
In weight diagrams (i) and (iii), there are no arrow pairs. In weight diagram (ii) there is exactly one arrow pair, given by the dots in position~$1$ and~$3$. In weight diagram (iv) there are two arrow pairs, one corresponds to positions $\minus 1$ and~$3$, the other to $\minus 1$ and~$1$.
\end{ex}

\subsubsection{Arrow diagrams} 
Similarly to~\cite[Section~6.2]{gang}, the arrow diagram for a partition~$\lambda$ consists of the weight diagram, decorated with an arrow from each white dot to each black dot which together form an arrow pair. It follows from the definition of arrow pairs that two arrows in an arrow diagram do not intersect, except possibly in the source (the starting white dot). This can be proved explicitly as in \cite[Lemma~6.2.2]{gang}.

We call pairs of arrows $a$ and $b$ such that source and target of~$a$ lie strictly between the source and target of~$b$ a \emph{nested} pair of arrows. A pair of arrows $a$ and $b$ which are not nested and have different sources is called a \emph{disjoint} pair of arrows. By the above, a pair $a$ and $b$ of disjoint arrows is automatically a pair such that the source and target of~$a$ are either both left of, or both right of, the source and target of~$b$.

\begin{ex}\label{ExDiag2}${}$
The non-trivial arrow diagrams corresponding to the weight diagrams in Example~\ref{ExDiag} are given by
\begin{enumerate}\item[(ii)] 
$ \!\!\!\!\!\!\!\!\!\!\!\!\!\!\!\! \xymatrix{  & \cdots &\underset{-4}{\bullet}  &\underset{-3}{\bullet} &\underset{-2}{\bullet}  &\underset{-1}{\bullet}  &\underset{0}{\circ}  &\underset{1}{\circ}\ar@{-->}@/_0.9pc/[rdru]  &\underset{2}{\circ} &\underset{3}{\bullet}  &\underset{4}{\circ} &\cdots} $
\vspace{2mm}
\item[(iv)]
$ \!\!\!\!\!\!\!\!\!\!\!\!\!\!\!\! \xymatrix{  & \cdots &\underset{-4}{\bullet}  &\underset{-3}{\bullet} &\underset{-2}{\bullet}  &\underset{-1}{\circ}\ar@{-->}@/_0.9pc/[rdru]\ar@{-->}@/_1.4pc/[rrdrru]  &\underset{0}{\circ}  &\underset{1}{\bullet}  &\underset{2}{\circ} &\underset{3}{\bullet}  &\underset{4}{\circ} &\cdots} $
\end{enumerate}
\end{ex}

\subsubsection{} For a weight diagram $x_\lambda$, the operation of constructing a new weight diagram by exchanging two dots which constitute a wb pair will be referred to as {\em flipping a wb pair}.

For a partition~$\lambda$ we define a set of partitions $\Pi(\lambda)$. We have $\nu\in \Pi(\lambda)$ if and only if $x_\nu$ can be obtained from~$x_\lambda$ by moving a number of white dots along arrows in the arrow diagram of~$\lambda$ (while the black dot travels in opposite direction). Equivalently, we can say that~$x_\nu$ can be obtained flipping a number of arrow pairs in $x_\lambda$. Note that the arrows along which boxes will be moved are by assumption automatically such that any two arrows are either nested or disjoint.

\begin{ex}
For $\lambda=(3,2)$, the two weight diagrams of the partitions in~$\Pi(\lambda)$ are given by
$$ \xymatrix{  & \cdots &\underset{-4}{\bullet}  &\underset{-3}{\bullet} &\underset{-2}{\bullet}  &\underset{-1}{\bullet}  &\underset{0}{\circ}  &\underset{1}{\circ}  &\underset{2}{\circ} &\underset{3}{\bullet}  &\underset{4}{\circ} &\cdots} $$
and
$$ \xymatrix{  & \cdots &\underset{-4}{\bullet}  &\underset{-3}{\bullet} &\underset{-2}{\bullet}  &\underset{-1}{\bullet}  &\underset{0}{\circ}  &\underset{1}{\bullet}  &\underset{2}{\circ} &\underset{3}{\circ}  &\underset{4}{\circ} &\cdots} $$
Hence, $\Pi(\begin{minipage}{0.97cm} \scalebox{.7}{\yng(3,2)} \end{minipage})=\{\begin{minipage}{1cm} \scalebox{.7}{\yng(3)} \end{minipage},\begin{minipage}{.45cm} \scalebox{.7}{\yng(1)} \end{minipage}\}$.
\end{ex}

\subsection{Decomposition multiplicities in terms of arrow diagrams}
Now we are ready to derive a fourth description of the set $\Gamma$. The following proposition implies that decomposition multiplicities for periplectic Brauer algebras can be described in terms of the arrow diagram calculus.
\begin{prop}\label{PropGang}
For two partitions $\lambda,\mu$, we have $\lambda\in \Pi(\mu)$ if and only if $\lambda\subseteq\mu$ with $\mu/\lambda\in\Gamma$.
\end{prop}
We start the proof with the following lemma.

\begin{lemma} \label{lemarrowhook}
Let $\mu$ be any partition and $x_\mu$ its weight diagram. Flipping a wb pair of dots in positions $(l_w,l_b)$, with $l_w < l_b$, yields a weight diagram $x_\lambda$, where $\lambda$ is obtained from~$\mu$ by removing a rim hook~$\gamma$ such that
\begin{enumerate}[(i)]
\item $\htt(\gamma)$ is the number of black dots in $x_\mu$ in the interval $[l_w,l_b]$;
\item we have $\wdd(\gamma)=l_b-l_w-\htt(\gamma)+1$;
\item with $a$ the anticontent and $c$ the content of the minimal box in $\gamma$, the anticontent of the box with content $c+i$ is given by
$$a+i-2\sharp\{\mbox{black dots in }[l_w+1, l_w+i]\},\quad\mbox{for $0\le i\le t+s-2$}.$$
\end{enumerate}
\end{lemma}
\begin{proof}
Let $s$ denote the number of black dots in $x_\mu$ in the interval $[l_w,l_b]$.
Denote by~$r_{s-1}< \ldots < r_0$ the positions of these black dots, hence $r_0=l_b$. For simplicity we set $r_s = l_w$. By definition of~$x_\mu$ and with $j$ the number of black dots right of~$l_b$, we have 
$$\mu_{j+i} = r_i + {j+i-1}, \qquad\mbox{for}\quad 0 \leq i < s. $$

Moving the black dot from~$r_0$ to $r_s$ can equivalently be interpreted as changing the position of all the indicated black dots from~$r_i$ to $r_{i+1}$. Thus $\lambda_{j+i} = r_{i+1} + j + i - 1$, or
\begin{equation}\label{eqlm}\lambda_{j+i}\,=\,\mu_{j+i}-(r_i-r_{i+1}),\qquad\mbox{for }\; 0\le i<s\end{equation} and $\lambda_q=\mu_q$ for all other $q$. Thus $\lambda \subseteq \mu$ and the skew diagram $\mu / \lambda$ has boxes in $s$ rows. It remains to be checked that the skew shape $\mu / \lambda$ is indeed a hook with the right properties. 

Comparing neighbouring rows we immediately have
\begin{equation*}\label{eqp1}
\mu_{j+i+1} = r_{i+1} + j + i \,=\, \lambda_{j+i} + 1,\qquad\mbox{for }\; 0\le i< s-1.
\end{equation*}
Thus the skew diagram is a connected hook~$\gamma:=\mu/\lambda$. By the above, we already know $\htt(\gamma)=s$, proving (i). By equation~\eqref{eqlm}, the total number of boxes in the skew shape $\mu/\lambda$ is equal to
\begin{equation*}\label{eqinterval}
\sum_{i=0}^{s-1} r_i - r_{i+1} = r_0 - r_s = l_b - l_w.
\end{equation*}
Part~(ii) then follows from equation~\eqref{sizehook}.

For the box with content $c+i$, the anticontent is given by $a+i$ minus twice the number of rows the box is above the minimal one,
which proves part~(iii).
\end{proof}

\begin{cor}\label{CorPair}
Fix a partition~$\mu$ with weight diagram $x_{\mu}$. Removing a rim hook from~$\mu$ is equivalent to flipping a wb pair in $x_\mu$. 

Consider such a rim hook~$\gamma$ with corresponding wb pair $p$.
 The wb pair $p$ is an arrow pair if and only if $\gamma$ is in $\Gamma_0$.
\end{cor}
\begin{proof}
Lemma~\ref{lemarrowhook} implies in particular that flipping a wb pair of dots in $x_\mu$ corresponds to removing a rim hook from~$\mu$. Using the same arguments one immediately finds that any rim hook can be obtained this way, {\it i.e.} removing a rim hook always corresponds to flipping a wb pair.

So assume that we flip a wb pair $p$ in positions $(l_w,l_b)$ in $x_\mu$.
By Lemma \ref{lemarrowhook}(i) and (ii) it follows that the rim hook~$\gamma$ satisfies
$$\wdd(\gamma)-\htt(\gamma)-1\;=\; l_b-l_w-2\;\sharp\{\mbox{black dots in $x_\mu$ in $[l_w,l_b]$}\}.$$
Hence, $\gamma$ satisfies the HW-condition if and only if $p$ satisfies the hw-condition.
Lemma \ref{lemarrowhook}(iii) implies that~$\gamma$ satisfies the D-condition if and only if $p$ satisfies the d-condition. 
\end{proof}

\begin{lemma}\label{LemDisjointNested}
Fix a partition~$\mu$ with weight diagram $x_{\mu}$. Flipping two disjoint wb pairs corresponds to removing two disjoint rim hooks. Flipping two nested wb pairs corresponds to removing two nested rim hooks.
\end{lemma}
\begin{proof}
Denote the resulting partition after removing the rim hooks by~$\lambda$. 

Assume first that the two wb pairs $p_1$ and $p_2$ are disjoint and that $p_2$ is right of $p_1$. The black dots in the two intervals spanned by the two pairs contain pairwise different sets of black dots, which implies that their respective rim hooks occupy pairwise disjoint sets of rows. In addition, the left most black dot involved in the arrow pair $p_2$, which corresponds to row $j\in\mZ_{>0}$ for some $j$, will be moved to the right of the right most black dot in $p_1$, corresponding to some row $k$ with $k>j$. Thus we have
$$\lambda_j-(j-1) > \mu_{k}-(k-1) + (k-j-1),$$ 
with $(k-j-1)$ the number of black dots in between the two. This implies implies $\lambda_j + 1> \mu_k$, which translates to the fact that the columns occupied by the rim hook corresponding to $p_2$ are strictly bigger than the columns occupied by the rim hook for $p_1$

Assume now that $p_1$ is nested in $p_2$ and we first flip $p_2$, resulting in a partition~$\nu$, and then $p_1$, resulting in $\lambda$. Then the rim hook corresponding to $p_1$ occupies a subset of the set of rows of the rim hook corresponding to $p_2$ by construction. This in turn already implies that the rim hook for $p_1$ cannot occupy a column that is strictly larger than the columns occupied by the rim hook for $p_2$. In addition the left most black dot of $p_1$, again assuming this corresponds to a row $j$, is moved to the right of the white dot of the arrow pair $p_2$ and to the right of the black dot corresponding to row $j+q$. Then it holds $\lambda_j-(j-1) > \nu_{j+1}- j$. Hence $\lambda_j+1 > \nu_{j+1}$, which implies that the set of columns occupied by the rim hook for $p_1$ is a subset for the set of columns occupied by the rim hook for $p_2$. That the rim hooks are nested thus follows from construction and the above considerations about the rows and columns they occupy.
\end{proof}

\begin{proof}[Proof of Proposition~\ref{PropGang}]
First assume that $\lambda \in \Pi(\mu)$. Denote the arrow pairs in $x_\mu$ which are flipped to create $x_\lambda$ by~$p_1,\cdots,p_k$, which are labelled such that $l_w^1< l_w^2<\cdots<l_w^k$ with $l_w^i$ the position of the white dot in $p_i$. Flipping the pair $p_1$ yields a weight diagram $x_\nu$ for a partition~$\nu$ with $\mu/\nu\in \Gamma_0$ by Corollary~\ref{CorPair}. By assumption any two arrows are either disjoint or nested. This implies in particular that $p_2,\cdots,p_k$ are still arrow pairs in $x_\nu$. We can thus proceed iteratively and we find that $\lambda$ is obtained from $\mu$ by consecutively removing rim hooks which are in $\Gamma_0$. The fact that all these hooks are either disjoint or nested then follows from Lemma~\ref{LemDisjointNested}.

Now assume that~$\lambda \subseteq \mu$ and $\mu / \lambda \in \Gamma$. Note that we can obtain $\lambda$ from~$\mu$ by successively removing hooks from the covering of~$C(\mu / \lambda)$, obtaining a partition in each step. Each of these steps will thus correspond to flipping an arrow pair by Corollary~\ref{CorPair}. As above it follows that all these pairs in intermediate weight diagrams are also arrow pairs in $x_\mu$. 
Thus one immediately obtains that~$\lambda \in \Pi(\mu)$.\end{proof}

\subsection{Multiplicity one property for projective modules}

In this section we show that the multiplicities $[P_r(\nu):L_r(\mu)]$, which are determined by equation~\eqref{eqPL} and Theorem~\ref{MainThm}, are either $1$ or $0$. We obtain this from results in \cite{gang} which we can now translate to our setting by Proposition~\ref{PropGang}.

\begin{prop}\label{PropMult1}
For any $\mu,\nu\in\Lambda_r$, we have
$$[P_r(\nu):L_r(\mu)]\;\le\; 1.$$
\end{prop}

Before proving this we introduce some notions related to the periplectic Lie superalgebra~$\mathfrak{pe}(n)$ for $n\in\mZ_{\ge 3}$. This is a Lie superalgebra with underlying Lie algebra~$\mathfrak{gl}(n)$. We follow the conventions of \cite{gang}, so in particular take the corresponding triangular decomposition of~$\mathfrak{pe}(n)$, with Cartan subalgebra~$\fh\cong\mk^n$ and $\fh^\ast=\langle\varepsilon_1,\ldots,\varepsilon_n\rangle_{\mk}$. We set
$$X_n:=\{\omega=\sum_{j=1}^n\omega_j\varepsilon_j\,|\, \omega_j\in\mZ \;\;\mbox{ and }\;\; \omega_1\ge \omega_2\ge\cdots\ge \omega_n\}\;\;\subset\;\;\fh^\ast.$$ We work in the category $\cF$ of finite dimensional integrable modules. The simple modules are given by~$S(\omega)$ (up to parity), with highest weight $\omega\in X_n$. Following \cite[Section~3.1]{gang}, we have the thick Kac module $\Delta(\omega)$, for any $\omega\in X_n$. We also denote the projective cover of~$S(\omega)$ in $\cF$ by~$Q(\omega)$. This module has a filtration with sections given by thick Kac modules and the corresponding multiplicities $(Q(\omega):\Delta(\xi))$ do not depend on the choice of filtration.
For any partition~$\lambda$ with length bounded by~$ n$, we associate
$$\underline{\lambda}\;=\;\sum_{j=1}^n\lambda_j\varepsilon_j\;\in X_n.$$
By the reformulation of our main result into arrow diagram combinatorics, we see that cell multiplicities for the periplectic Brauer algebra are special cases of Kazhdan-Lusztig multiplicities for the periplectic Lie superalgebra as determined in \cite{gang}.
\begin{lemma}\label{LemGang}
Assume $r\le n$, $\lambda\in\LL_r$ and $\mu\in \Lambda_r$, then we have
\begin{enumerate}[(i)]
\item $[W_r(\lambda):L_r(\mu)]\;=\; [\Delta(\underline{\lambda}):S(\underline{\mu})]$;
\item $[W_r(\lambda'):L_r(\mu')]\;=\; (Q(\underline{\mu}):\Delta(\underline{\lambda}))$.
\end{enumerate}
\end{lemma}
\begin{proof}
By Theorem~\ref{MainThm} and Proposition~\ref{PropGang}, we have
\begin{equation}\label{eqWgang}[W_r(\lambda):L_r(\mu)]=\begin{cases}1& \lambda\in\Pi(\mu),\\ 
0&\mbox{otherwise.}\end{cases}\end{equation}
On the other hand, by \cite[Theorem~6.3.3]{gang}, we have
$$[\Delta(\underline\lambda):S(\underline\mu)]=\begin{cases}1&\mbox{if }\; \underline\lambda\in\blacktriangledown (\underline\mu),\\ 
0&\mbox{otherwise,}\end{cases}$$
with $\blacktriangledown(-)$ defined in \cite[Section~6.2]{gang}. It follows immediately that, since the weight diagram of~$\underline\mu$ in \cite[Section~6.2]{gang} is identical to $x_\mu$ except that all dots in positions in $\mZ_{\le \minus n}$ are changed from black to white, the conditions $ \lambda\in\Pi(\mu)$ and $ \underline\lambda\in\blacktriangledown (\underline\mu)$ are equivalent. (However, we stress that in general $\blacktriangledown (\underline\mu)$ will have higher cardinality than $\Pi(\mu)$, but the `extra' elements are not of the form $\underline\lambda$ with $\lambda\in\LL_r$.) This proves part~(i).

By \cite[Theorem~6.3.1]{gang}, we have
$$(Q(\underline{\mu}):\Delta(\underline{\lambda}))\;=\;\begin{cases} 1& \mbox{if }\; \underline\lambda\in\blacktriangle (\underline\mu),\\
0&\mbox{otherwise}.\end{cases}$$
Using the description of~$x_{\mu'}$ in~\ref{xtrans} it follows that~$\underline\lambda\in\blacktriangle (\underline\mu)$ if and only if $\lambda'\in\Pi(\mu')$. Part~(ii) then follows from equations~\eqref{eqWgang} and~\eqref{eqBGG}.
\end{proof}

\begin{proof}[Proof of Proposition~\ref{PropMult1}]
Equation~\eqref{eqPL} and Lemma~\ref{LemGang} imply that
\begin{eqnarray*}[P_r(\nu):L_r(\mu)]&=& \sum_{\lambda\in\LL_r} ( Q(\underline{\nu}):\Delta(\underline{\lambda}))[\Delta(\underline{\lambda}):S(\underline{\mu})]\\\
&\le &  \sum_{\omega\in X_n} ( Q(\underline{\nu}):\Delta(\omega))[\Delta(\omega):S(\underline{\mu})]\;=\;[Q(\underline{\nu}):S(\underline{\mu})].\end{eqnarray*}
The conclusion thus follows from \cite[Theorem~8.1.2]{gang}.
\end{proof}

\subsection*{Acknowledgement}
This research was supported by the Australian Research Council grants DP140103239 and DP150103431.

\end{document}